\documentclass[11pt,letterpaper]{amsart}
\usepackage[utf8]{inputenc}
\usepackage[english]{babel}
\usepackage{amsmath,amssymb,amsthm,mathrsfs,color,times,textcomp,yfonts,mathtools,cases}
\usepackage[top=3.5cm,bottom=3.5cm,left=3cm,right=3cm]{geometry}
\allowdisplaybreaks[4]

\usepackage{bm}

\usepackage[T1]{fontenc} 

\usepackage{subcaption}
\usepackage[normalem]{ulem}
\usepackage[export]{adjustbox}
\usepackage{esint}
\usepackage{xcolor}
\usepackage{array}
\usepackage{enumitem}
\usepackage[colorlinks=true]{hyperref}
\hypersetup{urlcolor=blue, citecolor=blue, linkcolor=red}

\hypersetup{
	colorlinks=true,
	linkcolor=red
}

\usepackage[square,numbers]{natbib}

\usepackage{float}
\usepackage{ulem}
\usepackage{syntonly}
\usepackage{mathtools}
\usepackage{bm}
\usepackage{soul}
\usepackage{amsfonts,amsmath,latexsym,verbatim,amscd,mathrsfs,color,array}
\usepackage[colorlinks=true]{hyperref}

\usepackage{amsmath,amssymb,amsthm,amsfonts,graphicx,color}
\usepackage{amssymb}
\usepackage{epstopdf}

\newcommand{\R}{ \mathbb{R} }
\newcommand{\ts}{ \tilde s }
\newcommand{\tth}{ \tilde \theta }
\newcommand{\tg}{ \tilde g }
\newcommand{\tG}{ \tilde G }
\newcommand{\tr}{ \tilde r }
\newcommand{\Ge}{ \Gamma_{\epsilon
} }
\newcommand{\m}{\mathbf{m} }
\newcommand{\n}{ \mathbf{n} }

\newtheorem{definition}{Definition}
\theoremstyle{plain}
\newtheorem{theorem}{Theorem}[section]
\newtheorem{lemma}[theorem]{Lemma}

\newtheorem{proposition}[theorem]{Proposition}

\newtheorem{remark}{Remark}[theorem]
\newcommand{\bremark}{\begin{remark} \em}
	\newcommand{\eremark}{\end{remark} }

\numberwithin{equation}{section}

\begin{document}
\title{Stable solutions of $U(1)$ Yang-Mills-Higgs model in $\R^4$}

\author[Yong Liu]{Yong Liu}
\address{Yong Liu, School of Mathematics and Statistics, Beijing Technology and
Business University, Beijing, China.}
\email{yliumath@btbu.edu.cn}
\author[Juncheng Wei]{Juncheng Wei}
\address{Juncheng Wei, Department of Mathematics, Chinese University of Hong
Kong, Shatin, NT, Hong Kong.}
\email{wei@math.cuhk.edu.hk}
\author[Zikai Ye]{Zikai Ye}
\address{Zikai Ye, Department of Mathematics, Chinese University of Hong
Kong, Shatin, NT, Hong Kong.}
\email{zkye@math.cuhk.edu.hk}

\maketitle

\begin{abstract}
We give a positive answer to the conjecture of Liu-Ma-Wei-Wu in \cite{LMWW} that the family of entire solutions to the $U(1)$-Yang-Mills-Higgs equation constructed by the gluing method in that paper are stable. This is the first family of examples of nontrivial stable critical points to the $U(1)$-Yang-Mills-Higgs model in higher dimensional Euclidean space. Intuitively, the stability of these solutions corresponds to the fact that holomorphic curves are area-minimizing. We also show that these entire solutions are non-degenerate. Our proof is based on detailed analysis of the linearized operators around this family and the spectrum estimates of the Jacobi operator by Arezzo-Pacard \cite{ArePac}. 
\end{abstract}

\section{Introduction and statement of main result}
Yang-Mills-Higgs type functionals and their associated Euler-Lagrange equations are among the most important models in modern physics. The study of these models, for instance, the instanton and monopole equations\cite{Donaldson, Taubes}, has triggered many fascinating mathematical theories and results. 

Here we are interested in a Yang-Mills-Higgs model with Ginzburg-Landau type potential and $U(1)$ gauge group (also called magnetic Ginzburg-Landau equation, in low dimensions):
\begin{equation}
 \label{YMH}
 \mathcal E(\psi,A):=\int_{M}\left\{  \left\vert \nabla_{A}\psi\right\vert ^{2}+\left\vert dA\right\vert
^{2}+\frac{\lambda}{4}\left(  1-\left\vert \psi\right\vert ^{2}\right)
^{2}\right\}.
\end{equation}
Recently there has been increasing interest in this action functional. Lin-Rivière, Parise-Pigati-Stern, Pigati-Stern \cite{LR,Parise, Pigati-Stern} studied the asymptotic behavior of critical points of this functional in the self-dual case. Among other things, they showed that a sequence of solutions with uniformly bounded energies will converge in a suitable sense to a codimension two integral varifold. When the solutions are minimizers, the resulting varifold will also be area-minimizing. This extends the previous result of Hong-Jost-Struwe \cite{HJS} for Riemann surfaces and  Bradlow \cite{B} for K\"ahler manifolds.

On the opposite side, using a variational argument, De Pilippis-Pigati \cite{Philippis-Pigati} proved the existence of a family of solutions concentrating on given non-degenerate minimal submanifolds of codimension two. Presumably, the Morse indices of these solutions should be related to that of the minimal submanifolds. Note that in the self-dual case, classification results for entire solutions have obtained by Taubes \cite{T2,T1} for finite energy critical points in 2d and De Philippis-Halavati-Pigati in \cite{PHP}, which states that local minimizers satisfying suitable energy growing estimates have to be trivial. At this point, it is worth pointing out that in the non-self dual case,  Rivi\`ere \cite{Riv} proved in the 2d case that local minimizers are unique up to gauge transformation, provided that the coupling constant $\lambda$ is large enough.

The aforementioned results and their proof are indeed partly inspired that of the $\Gamma$-convergence and related results for the Allen-Cahn functional: \[
\int\left\{  \left\vert \nabla u\right\vert ^{2}+\frac{1}{2}\left(
1-u^{2}\right)  ^{2}\right\},
\]where $u$ is a scalar function. The Allen-Cahn equation
\begin{equation}
\label{AllenCahn}
\Delta u+u-u^3=0 \ \ \mbox{in} \ \R^n, 
\end{equation}
which is the Euler-Langrange equation of this functional, can be used as a regularization of codimension one minimal submanifolds. This point of view turns out to be very useful and has some important applications in the minimal surface theory, see for instances \cite{CM1,CM2, Dey, GG, Gu}. It should also be emphasized that the codimension two case imposes more technical difficulties than the codimension one case.  

To better state our main results, let us mention some other  interesting nontrivial results  obtained so far for the codimension one case. The celebrated De Giorgi conjecture \cite{DeG} states that for any bounded entire solution $u$ to the Allen-Cahn equation (\ref{AllenCahn}) which is strictly monotone in one direction ($x_n$-direction for example) should be one-dimensional. This is established in dimension $2$ by Ghoussoub-Gui \cite{GG1} and in dimension $3$ by Ambrosio-Cabr\'e \cite{AC}. Partial results are obtained by Ghoussoub-Gui \cite{GG2} in dimension $4$ and $5$. Under the additional assumption that
\begin{equation*}
\lim_{x_{n}\to \pm \infty}u(x^\prime,x_n)=\pm 1,\text{ for all }x^\prime \in \R^n,
\end{equation*}
Savin \cite{S} proves the affirmative of the De Giorgi conjecture for $4\leq n \leq 8$. On the other hand,  Gluing techniques turn out to be  powerful in the construction of entire solutions of the Allen-Cahn equation. As a matter of fact, counterexamples of the De Giorgi conjecture in dimension $9$ have been constructed by Del Pino-Kowalczyk-Wei \cite{PinoKow1}, using the Bombieri-De Giorgi-Giusti minimal graphs.

Later on,  entire stable solutions to the Allen-Cahn equation  in dimension $n\geq 8$ are also constructed in \cite{PW} based on the family of minimal submanifolds asymptotic to the famous Simons' cone. These solutions are shown to be global minimizers in Liu-Wang-Wei \cite{LWW1}. Other stable solutions are saddle solutions on Simons' cones (Cabr\'e-Terra \cite{CT1, CT2}). They are stable in higher dimensions (Cabr\'e \cite{Ca}, Liu-Wang-Wei \cite{LWW2}).   The key ingredient in the proof of stability is the existence of a positive kernel $\phi>0$ of the linearized operator $\Delta+1-3u^2$ at the solution $u$. Once such $\phi$ exists, then the stability follows from testing the Allen-Cahn equation by $\phi^{-1}\psi^2$ to get
\begin{equation}
\label{sta}
\int_{\R^n} |\nabla\psi|^2-\psi^2 +3u^2\psi^2=\int_{\R^n}|\nabla\psi-\phi^{-1}\psi\nabla\phi|^2\geq 0
\end{equation}
for any compactly supported test function $\psi$. However, a similar argument using (\ref{sta}) does not work for the stability of the Yang-Mills-Higgs equation that we will discuss later, since the solution is vector-valued and is not simply a scalar function. We also would like to mention that in $\R^3$, Del Pino-Kowalczyk-Wei \cite{PinoKow2} constructed a family of finite Morse index solutions to the Allen-Cahn equation that concentrated near complete, embedded, non-degenerate minimal surface with finite total curvature. The Morse indices of these solutions coincide with the concentrated minimal surfaces. In fact, all these solutions constructed by gluing techniques exhibit suitable concentration near codimension one minimal submanifolds, while in the normal direction, they look like the standard one-dimensional heteroclinic solution of the Allen-Cahn equation.

Let $\Delta_A:=\nabla^*_A \nabla_A$ be the connection Laplacian. For the Yang-Mills-Higgs model (\ref{YMH}) in the Euclidean space with a trivial Hermian bundle, the Euler-Lagrange equation has the form
\begin{equation}
\label{GLE}
\begin{cases}
    -\Delta_A \psi +\frac{\lambda}{2}(|\psi|^2-1)\psi=0\text{ in }\R^n,\\
    d^*dA-Im(\nabla_A\psi\cdot \bar\psi)=0\text{ in }\R^n.
\end{cases}
\end{equation}
Naturally, in view of the developments for the Allen-Cahn function, one expects to be able to build solutions to $(\ref{GLE})$ based on the standard vortex solutions (See Section 2 for more precise form of these solutions) in the two-dimensional plane, in replace of the one dimensional heteroclinic solution to the Allen-Cahn equation. The stability(or instability) of these 2d vortex solutions depends on the parameter $\lambda$ and the degree of the solutions and is resolved in Gustafson-Sigal \cite{GS}. A quantitative stability for critical points is proved by Halavati \cite{H1,H2}. Detailed analysis, including the $\Gamma$-convergence theory, of solutions to the equation \eqref{GLE} in dimension two is discussed in Sandier-Serfaty \cite{SS}. Actually, the classical Ginzburg-Landau equation has also been  discussed there. The main difference between the theory of the Ginzburg-Landau equation and \eqref{GLE} is their asymptotic behavior at far field. Vortex solutions to \eqref{GLE} decay exponentially fast, while that of the Ginzburg-Landau only decays at an algebraic rate, implying that they are much more difficult to deal with.  Here we will not touch on the classical Ginzburg-Landau equation, only refer to \cite{SS} and the references therein for more discussion, see also \cite{BBH,PacRiv}.

At this stage, we already know that entire solutions of the Yang-Mills-Higgs model with multiple vortex points can been constructed in \cite{TW} in the two dimensional plane. In higher dimensions, Brendle \cite{Br} and Badran-Del Pino  \cite{BM1,BM2,BM3} are able to  use gluing construction to build manuy interesting solutions concentrated on codimension 2 minimal submanifolds based on the aforementioned 2d vortex solutions. 

Since holomorphic curves are area-minimizing,  many codimension two stable minimal submanifolds already exist in $\R^4$. In \cite{LMWW}, a family of entire solutions is constructed using Lyapunov-Schmidt reduction arguments, they concentrated on suitable rescaling of the codimension 2 minimal submanifold studied by Arezzo-Pacard in \cite{ArePac}, given by 
\begin{equation}
\label{mininalAP}
    \Gamma =\frac{e^{is}}{\sqrt{\sin 2s}}\Theta,
\end{equation}
where $s\in(0,\frac\pi 2), \Theta=(\cos\theta,\sin\theta)\in \mathbb S^1, \theta\in [0,2\pi)$.
Geometrically, this manifold has two planar ends and can be regarded as a desingularization of the union of two orthogonal planes. Let us denote these solutions by $U_{\epsilon}$, where $\epsilon>0$ is a sufficiently small rescaling parameter. Here we show that they are stable critical points of the Yang-Mills-Higgs model (\ref{YMH}), with the manifold $M$ being the four-dimensional Euclidean space $\R^4$. Our main result can be stated in the following

\begin{theorem}
\label{main}
The solutions $U_{\epsilon}$ are stable and non-degenerate. 
\end{theorem}
The notion of stable means that the quadratic form associated to the Yang-Mills-Higgs functional is always nonnegative. In view of its counterpart in the Allen-Cahn case \cite{PinoKow2}, intuitively, this should be true due to the fact that the concentrating minimal submanifold is area minimizing. However, the proof will involve many technical details. The main step is to get precise asymptotic behavior of the solutions. This will then enable us to reduce the stability of solutions to the analysis of the Jacobi operator. It is expected that this family of solutions are also minimizers of the functional, but this problem seems to be much more difficult. We believe that this strategy can be applied to prove that the Morse index of some critical points of $U(1)$-Yang-Mills-Higgs functional are determined by its concentrated minimal submanifolds, as in the Allen-Cahn case.

The non-degeneracy of $U_\epsilon$ means that any bounded kernel of linearized operator $L(\cdot;U_\epsilon)$ at $U_\epsilon$ is a linear combination of $Z_j$, $j=1,\cdots, 6$(corresponding to translation or rotation invariance of the equation) and gauge kernels(which arise from the gauge invariance) to be introduced more precisely in Section 2.

The rest of the paper is devoted to proving Theorem \ref{main}. It is organized in the following way: In Section 2, we recall some preliminary results, including the properties of the minimal submanifold $\Gamma$, the vortex solution and related apriori estimates. In Section 3,  we analyze the linearized operator $\mathbb L(\cdot;\psi,A)$ in detail and build up its relationship between the Jacobi operator $L_\Gamma$ of $\Gamma$. Section 4 finishes the proof of  Theorem \ref{main}.

{\center{\bf Acknowledgement}.} Y. Liu is supported by the National Key R\&D Program of China 2022YFA1005400 and NSFC 12471204. J.C. Wei is supported by National Key R\&D Program of China 2022YFA1005602, and Hong Kong General Research Fund ``New frontiers in singular limits of nonlinear partial differential equations''. The third author thanks Wangzhe Wu for useful discussion.

\section{Preliminaries}
In this section, we collect some facts which will be used later, including the stability of the minimal submanifold $\Gamma$, Fermi coordinates, and related computations. We also list the properties of 2d vortex solutions and the vortex solutions constructed in \cite{LMWW}. Apriori estimates of the eigenvalue problems and linear theory for the linearized operator are also briefly discussed.

\subsection{The minimal submanifold $\Gamma$}
Recall that the minimal submanifold $\Gamma$ defined by (\ref{mininalAP}) is of codimension two. It is a holomorphic curve in $\mathbb{C}^2$. Hence, it is area-minimizing, in particular, it is stable. That is,
\begin{proposition}
$\Gamma$ is a stable minimal submanifold.
\end{proposition}

After a rescaling by a small parameter $\epsilon$, we get a family of minimal submanifolds, all of them being far away from the origin, with the form
\begin{equation*}
    \Gamma_\epsilon =\frac{e^{i\epsilon\tilde s}}{\epsilon\sqrt{\sin 2\epsilon\tilde s}}\tilde \Theta,
\end{equation*}
where $\ts=\epsilon^{-1}s\in (0,\frac{\pi}{2\epsilon})$, $\tth=\epsilon^{-1}\theta\in [0,\frac{2\pi}{\epsilon})$ and $\tilde \Theta=(\cos \epsilon\tth, \sin \epsilon\tth)$.  \cite{LMWW} constructed a solution $(\psi_\epsilon,A_\epsilon)$ near $\Gamma_\epsilon$.  To describe the solutions in a more precise way, it will be necessary to use the Fermi coordinate, to be recalled in sequel.

Let $\m=ie^{-i\epsilon \ts}\tilde\Theta, \n=ie^{i\epsilon\ts}\tilde\Theta^\perp$ be the two unit normal vectors of $\Gamma_\epsilon$, where $\tilde \Theta^\perp=(-\sin\epsilon\tth,\cos\epsilon\tth)$. Define a map $T:(\ts,\tth,a,b)\to (z_1,z_2)\in \mathbb C^2$:
\begin{equation*}
\begin{aligned}
    Y&=\Gamma_\epsilon+a\m+b\n\\
    &=\left(\frac{\cos \epsilon \ts}{\epsilon\sqrt{\sin 2\epsilon\ts}}\tilde\Theta+a\sin \epsilon\ts\cdot \tilde\Theta-b\sin \epsilon\ts \cdot \tilde\Theta^\perp, \frac{\sin \epsilon\ts}{\epsilon\sqrt{\sin 2\epsilon\ts}}\tilde\Theta+a\cos \epsilon\ts\cdot \tilde\Theta+b\cos \epsilon\ts \cdot \tilde\Theta^\perp\right).
\end{aligned}
\end{equation*}
Write $\Sigma_\epsilon:=\{(\ts,\tth,a,b): a^2+b^2<\frac{1}{\epsilon\sin 2\epsilon\ts}=:r_{\epsilon}^2\}$. These provide a local coordinate system for the tubular neighborhood $\Sigma_\epsilon$, called the Fermi coordinates. We remark that actually, this local coordinate system is well-defined in $\{(\ts,\tth,a,b): a^2+b^2<\epsilon^{-1}r_{\epsilon}^2\}$. But $\Sigma_\epsilon$ is sufficient for later purposes. 

In the Fermi coordinates, the  metric is given by
\begin{equation*}
    (g_{ij})=
    \begin{pmatrix}
        \tilde a_{11}& \tilde a_{12}& 0& 0\\
        \tilde a_{21}& \tilde a_{22}& -b\epsilon\cos (2\epsilon\ts)& a\epsilon\cos (2\epsilon\ts)\\
        0& -b\epsilon\cos (2\epsilon\ts)& 1& 0\\
        0& a\epsilon\cos (2\epsilon\ts)& 0& 1
    \end{pmatrix},
\end{equation*}
where
\begin{equation*}
    \tilde A:=(\tilde a_{ij})=
    \begin{pmatrix}
        \left(\epsilon a-\frac{1}{ (\sin 2\epsilon\ts)^{\frac32}}\right)^2+\epsilon^2 b^2& \frac{-2b\epsilon}{\sqrt{\sin 2\epsilon\ts}}\\
        \frac{-2b\epsilon}{\sqrt{\sin 2\epsilon\ts}}& \epsilon^2(a^2+b^2)+\frac{1}{\sin 2\epsilon\ts}+2a\epsilon\sqrt{\sin 2\epsilon\ts}
    \end{pmatrix}
\end{equation*}
is a metric on $\Gamma_{\epsilon}$. Its determinant will be denoted by $G=\det(g_{ij})$. One of the reasons that we will do computations in these coordinates, instead of some moving frames, is that all these functions are completely explicit.

For simplicity, we define $\varrho$ such that $\varrho(x):=(\sin 2s)^{-1}$ for $x\in \Sigma_1$ and extend it outside $\Sigma_1$ to be a smooth function on $\R^4$ with
\begin{equation*}
    c_1(1+|x|)\leq \varrho(x)\leq c_2(1+|x|)\text{ in }\R^4
\end{equation*}
for some constants $c_1,c_2>0$. Then the inverse of the metric matrix can be expanded as
\begin{equation*}
\begin{aligned}
    &(g^{ij})=\begin{pmatrix}
        \rho^{-6}+2a\epsilon\rho^{-9}& 2b\epsilon\rho^{-7}& 0& 0\\
        2b\epsilon\rho^{-7}& \rho^{-2}-2a\epsilon\rho^{-5}& b\epsilon\rho^{-2}\cos2\epsilon\ts& -a\epsilon\rho^{-2}\cos2\epsilon\ts\\
        0& b\epsilon\rho^{-2}\cos2\epsilon\ts& 1& 0\\
        0& -a\epsilon\rho^{-2}\cos2\epsilon\ts& 0& 1
    \end{pmatrix}\\
    &+\begin{pmatrix}
        3r^2\epsilon^2\rho^{-12}& 0& b^2\epsilon^2\rho^{-7}\cos2\epsilon\ts& -ab\epsilon^2\rho^{-7}\cos2\epsilon\ts\\
        0& 3r^2\epsilon^2\rho^{-8}& -2ab\epsilon^2\rho^{-5}\cos2\epsilon\ts& 2a^2\epsilon^2\rho^{-5}\cos2\epsilon\ts\\
        b^2\epsilon^2\rho^{-7}\cos2\epsilon\ts& -2ab\epsilon^2\rho^{-5}\cos2\epsilon\ts& b^2\epsilon^2\rho^{-2} \cos^2 2s& -ab\epsilon^2\rho^{-2} \cos^2 2s\\
        -ab\epsilon^2\rho^{-7}\cos2\epsilon\ts& 2a^2\epsilon^2\rho^{-5}\cos2\epsilon\ts& -ab\epsilon^2\rho^{-2} \cos^2 2s& a^2\epsilon^2\rho^{-2} \cos^2 2s
    \end{pmatrix}+O(\epsilon^3\rho^{-4}),
\end{aligned}
\end{equation*}
where $\rho(x):=\varrho(\epsilon x)$.

With these notations at hand, we can define the following inner products in normal space:
\begin{definition}
Let $(y_1,y_2,y_3,y_4)=(\ts,\tth,a,b)$. For any complex-valued functions $\zeta_1(\ts,\tth,a,b)$ and $\zeta_2(\ts,\tth,a,b)$, one-forms $C_1=C_{1i}dy_i$ and $C_2=C_{2j}dy_j$, we define the following inner products in normal space:
\begin{equation*}
\langle \zeta_1, \zeta_2 \rangle:=\Re(\zeta_1 \overline{\zeta_2}),\ \langle C_1,C_2\rangle:=g^{ij}C_{1i}C_{2j},
\end{equation*}
and
\begin{equation*}
\langle (\zeta_1,C_1),(\zeta_2,C_2)\rangle:=\langle \zeta_1, \zeta_2 \rangle+\langle C_1,C_2\rangle=\Re(\zeta_1 \overline{\zeta_2})+g^{ij}C_{1i}C_{2j}.
\end{equation*}

\end{definition}

Now let us turn to the Jacobi operator $L_\Gamma$ of the minimal submanifold $\Gamma$. For a normal vector field $N=ie^{-is}k_1\Theta+ie^{is}k_2\Theta^\perp=k_1\m+k_2\n$, we have
\begin{equation*}
\begin{aligned}
L_{\Gamma} N&=\Delta_{\Gamma}^{\nu}N+2\rho^{-6} N\\
&=[\rho^{-6} k_{1ss}+\rho^{-2} k_{1\theta\theta}+2\rho^{-4}\cos 2s k_{1s}-2\rho^{-2}\cos 2sk_{2\theta}+\rho^{-2}(2\rho^{-4}-\cos^2 2s)k_1]\m\\
&+[\rho^{-6} k_{2ss}+\rho^{-2} k_{2\theta\theta}+2\rho^{-4}\cos 2s k_{2s}+2\rho^{-2}\cos 2sk_{1\theta}+\rho^{-2}(2\rho^{-4}-\cos^2 2s)k_2]\n,
\end{aligned}
\end{equation*}
where $\Delta_{\Gamma}^{\nu}$ is the connection Laplacian on the normal bundle $\mathcal{N}\Gamma$ of $\Gamma$ in $\R^4$. See \cite{ArePac} for details.

A normal vector field $N=k_1\m+k_2\n$ is called a Jacobi field on $\Gamma$ if $L_\Gamma N=0$. We have the following bounded Jacobi fields generated by rigid motions:
\begin{equation}\label{Jacobifields}
\begin{aligned}
N_1&=\cos\theta\sin s\m+\sin\theta\sin s\n,\ N_2=\sin\theta\sin s\m-\cos\theta\sin s\n,\\ 
N_3&=\cos\theta\cos s\m-\sin\theta\cos s\n,\ N_4=\sin\theta\cos s\m+\cos\theta\cos s\n,\\
N_5&=(\sin 2s)^{\frac 12}\m,\ N_6=(\sin 2s)^{\frac 12} \n.
\end{aligned}
\end{equation}
Here $N_j$, $j=1,2,3,4$ are generated by the translation in $x_j$-direction. $N_5$ is generated by dilation and $N_6$ is generated by the action $x\mapsto Jx$ in $O(4)/SU(2)$, where
\begin{equation*}
J=\begin{pmatrix}
0 & -1 & 0 & 0\\
1 & 0 & 0 & 0\\
0 & 0 & 0 & 1\\
0 & 0 & -1 & 0
\end{pmatrix}.
\end{equation*}

It is shown in \cite{LMWW} that the submanifold $\Gamma$ is non-degenerate.
\begin{lemma}\label{Jacobinondeg}
$\Gamma$ is non-degenerate in the sense that any bounded Jacobi field on $\Gamma$ is a linear combination of $N_j$, $j=1,\dots,6$.
\end{lemma}

We also define the translated coordinate $t_1=a-\epsilon f_1, t_2=b-\epsilon f_2$ and the corresponding perturbed polar coordinate $(\tilde r, \tilde \phi)$, where $f_1(s,\theta)$ and $f_2(s,\theta)$ are perturbations given in Theorem 8.1 of \cite{LMWW} with a different notation $f_1(s,\theta)$ and $f_2(s,\theta)$. Moreover, for $j=1,2$, $f_j$  satisfy
\begin{equation*}
\|\nabla_\Gamma^2 f_j\|_{4,p}+\|\nabla_\Gamma f_j\|_{3,p}+\|f_j\|_{2,p}\leq C.
\end{equation*}
for some $p>1$. The norms will be introduced in \eqref{pbetanorm}.  

With the notation above, the connection Laplacian $\Delta_{\Gamma_{\epsilon}}^{\nu}$ and
\begin{equation*}
    \Delta_A\xi=\Delta\xi+id^* A\xi-2i\langle A,d\xi\rangle-|A|^2\xi
\end{equation*}
can be computed.

\begin{lemma}\label{Laplace1Expansion}
For a smooth normal vector field $N= k_1(s,\theta)m+k_2(s,\theta)n$ and smooth functions $\eta(s,\theta)$ on $\Gamma$ and $\xi(\ts,\tth,t_1,t_2)$ on $\Sigma_{\epsilon}$, we have
\begin{equation*}
\begin{aligned}
\Delta_{\Gamma}^{\nu}N
=&[\rho^{-6} k_{1ss}+\rho^{-2} k_{1\theta\theta}+2\rho^{-4}\cos 2s k_{1s}-2\rho^{-2}\cos 2sk_{2\theta}-\rho^{-2}\cos^2 2s k_1]\m\\
+&[\rho^{-6} k_{2ss}+\rho^{-2} k_{2\theta\theta}+2\rho^{-4}\cos 2s k_{2s}+2\rho^{-2}\cos 2sk_{1\theta}-\rho^{-2}\cos^2 2s k_2]\n,    
\end{aligned}
\end{equation*}

\begin{equation*}
\Delta_{\Gamma}\eta
=\rho^{-6} \eta_{ss}+\rho^{-2} \eta_{\theta\theta}+2\rho^{-4}\cos 2s \eta_{s}   
\end{equation*}

and
\begin{equation*}
\begin{aligned}
    \Delta\xi&=\xi_{t_1t_1}+\xi_{t_2t_2}+\Delta_{\Gamma_{\epsilon}}\xi+2\epsilon(t_2\xi_{\tth t_1}-t_1\xi_{\tth t_2})\rho^{-2}\cos (2\epsilon\ts)(1-2t_1\rho^{-3}\epsilon)\\
    &+\epsilon^2(t_2^2\xi_{t_1t_1}-2t_1t_2\xi_{t_1t_2}+t_1^2\xi_{t_2t_2})\rho^{-2}\cos^2 (2\epsilon\ts)-\epsilon^2\rho^{-2}(\rho^{-4}+1)(t_1\xi_{t_1}+t_2\xi_{t_2})\\
    &+2\epsilon^2 \rho^{-2}\cos (2\epsilon\ts) (f_2\xi_{t_1}-f_1\xi_{t_2})+4t_2\epsilon^2\rho^{-7}\cos2\epsilon\ts(t_2\xi_{\ts t_1}-t_1\xi_{\ts t_2})-2\epsilon^2\rho^{-6}(f_{1s}\xi_{\ts t_1}+f_{2s}\xi_{\ts t_2})\\
    &-2\epsilon^2\rho^{-2}(f_{1\theta}\xi_{\tth t_1}+f_{2\theta}\xi_{\tth t_2})+(2(t_1+\epsilon f_1)\rho^{-9}+3\tilde r^2\rho^{-12}\epsilon
    )\epsilon\xi_{\ts\ts}+4(t_2+\epsilon f_2)\epsilon\rho^{-7}\xi_{\ts\tth}\\
    &+(-2(t_2+\epsilon f_2)\rho^{-5}+3\tilde r^2\epsilon \rho^{-8})\epsilon\xi_{\tth\tth}+8t_1\epsilon^2\rho^{-7}\cos^2 (2\epsilon\ts) \xi_{\ts}+4t_2\epsilon^2\rho^{-5}\cos (2\epsilon\ts) \xi_{\tth}+O(\epsilon^3),
\end{aligned}
\end{equation*}

\begin{equation}\label{Laplace2Expansion}
\begin{aligned}
d^*A\xi
&=-\frac{1}{\sqrt{G}}\partial_j(\sqrt{G}g^{ij}A_j)
\xi+d^*B_0\xi+O(\epsilon^3)=d^*B_0\xi+O(\epsilon^3),
\end{aligned}
\end{equation}

\begin{equation}\label{Laplace3Expansion}
\begin{aligned}
\langle A,d\xi\rangle
&=g^{ij}A_i\xi_j+\langle B_0,d\xi\rangle+O(\epsilon^3)\\
&=A_3\xi_a+A_4\xi_b-\epsilon^2\rho^{-2}\cos^2 2s (b\xi_a-a\xi_b)-\epsilon^2 \rho^{-2} \cos 2s \xi_\theta\\
&-\epsilon^2\rho^{-6}(2t_2\rho^{-1}\cos (2\epsilon\ts)+A_3 f_{1s}+A_4 f_{2s})\xi_{\ts}+\epsilon^2\rho^{-2}(2t_1\rho^{-1}\cos (2\epsilon\ts)-A_3 f_{1s}-A_4 f_{2s})\xi_{\tth}\\
&+\langle B_0,d\xi\rangle+O(\epsilon^3),
\end{aligned}
\end{equation}

\begin{equation}\label{Laplace4Expansion}
\begin{aligned}
|A|^2\xi
&=g^{ij}A_iA_j\xi+(|A_0+B_0|^2-|A_0|^2)\xi+O(\epsilon^3)\\
&=(A_3^2+A_4^2)\xi+\epsilon^2\rho^{-2}\cos^2 2s \xi+(|A_0+B_0|^2-|A_0|^2)\xi+O(\epsilon^3).
\end{aligned}
\end{equation}
\end{lemma}

The following formulas can be found in \cite{BM1,BM2,BM3}.

\begin{lemma}\label{d*d}
For a smooth one-form $\omega=\omega_i dy^i$, we have
\begin{equation*}
d^*\omega=-\frac{1}{\sqrt{G}}\partial_j(\sqrt{G}g^{ij}\omega_i)
\end{equation*}
and
\begin{equation*}
d^*d\omega=-\frac{1}{\sqrt{G}}g_{ml}\partial_j(\sqrt{G}g^{ij}g^{kl}\omega_{ik})dy^m,
\end{equation*}
where $\omega_{ik}=\partial_i\omega_k-\partial_k\omega_i$.
\end{lemma}

\subsection{Vortex solutions}

In the two-dimensional plane, we have fundamental vortex solutions for each fixed integer $j$, called the topological degree of the solution. In the polar coordinate $(r,\phi)$, the $j$-vortex solutions $u_j=(\psi^{(j)},A^{(j)})$ in $\R^2$ has the form
\begin{equation*}
    \psi^{(j)}(x)=U_j(r)e^{\sqrt{-1}j\phi}\text{ and }A^{(j)}(x)=V_j(r)d(j\phi),
\end{equation*}
 where $U_j,V_j$ satisfy the following ODE system
\begin{equation}\label{ODEsystem}
\begin{cases}
    -U_j''-\frac1r U_j'+j^2\frac{(1-V_j)^2}{r^2}U_j-\frac\lambda 2 (1-U_j^2)U_j=0,\\
    -V_j''+\frac1r V_j'-U_j^2(1-V_j)=0.
\end{cases}
\end{equation}
Moreover, they have the following properties:
\begin{itemize}
    \item $0<U_j, V_j<1$ on $(0,+\infty)$.
    \item $U_j^\prime, V_j^\prime>0$.
    \item $U_j\sim c_1r$, $V_j\sim c_2r^2$ as $r\to 0$ for some constants $c_1,c_2>0$.
    \item $1-U_j, 1-V_j\to 0$ as $r\to \infty$ with an exponential rate of decay. In particular, $1-U_1=O(e^{-m_\lambda r})$, $1-V_1=O(e^{-r})$ as $r\to \infty$, where $m_\lambda=\min\{\sqrt\lambda, 2\}$.
\end{itemize}

The existence and above properties  of these functions are proved in \cite{BerChen,Plo}. From results of \cite{GS}, we know that for $|j|=1$, the vortex solution is always stable for any positive $\lambda$; while for $|j|>1$, it is stable when $\lambda\leq 1$ and unstable for $\lambda>1$. Since we are interested in stable solutions, here in this paper we will only use solutions with $|j|=1$. To simplify the notations, we set $U=U_1, V=V_1$, although they are actually also depending on the coupling parameter $\lambda$.

Roughly speaking, the family of vortex solutions $(\psi_\epsilon, A_\epsilon)$ constructed in \cite{LMWW} can be written into the form:
\begin{equation*}
    \psi_\epsilon=\psi_{0\epsilon}(s,\theta,a,b)+\eta_{0\epsilon}(s,\theta,a,b), A_\epsilon=A_{0\epsilon}(s,\theta,a,b)+B_{0\epsilon}(s,\theta,a,b),
\end{equation*}
where $(\psi_{0\epsilon},A_{0\epsilon})$ is suitable  approximate solution, and  $(\eta_{0\epsilon}, B_{0\epsilon})$ is small perturbation. For simplicity, we will omit the subscript $\epsilon$ later on. More precisely, $(\psi_0,A_0)$ can be written as
\begin{equation*}
\begin{aligned}
    &\psi_0(s,\theta,a,b)=\psi^{(1)}(a-\epsilon f_1,b-\epsilon f_2)=U(\tilde r)e^{i\tilde\phi},\\
    &A_0(s,\theta,a,b)=A_2(s,\theta,a,b)d\theta+A_3(t_1,t_2) dt_1+ A_4(t_1,t_2) dt_2\\
    =&(-\epsilon f_{1s} A_3-\epsilon f_{2s} A_4)ds+(A_2(s,\theta,a,b)-\epsilon f_{1\theta} A_3-\epsilon f_{2\theta} A_4)d\theta+A_3da+A_4db\\
    =&(\epsilon f_{1s}\frac{V}{\tilde r}\sin \tilde\phi-\epsilon f_{2s}\frac{V}{\tilde r}\cos\tilde\phi)ds+[-(1-V)\cos 2s+\epsilon f_{1\theta}\frac{V}{\tilde r}\sin \tilde\phi-\epsilon f_{2\theta}\frac{V}{\tilde r}\cos\tilde\phi]d\theta-\frac{V}{\tilde r}\sin\tilde\phi da+\frac{V}{\tilde r}\cos\tilde\phi db.
\end{aligned}
\end{equation*}
inside $\Sigma_{\epsilon}$, where $A_2(s,\theta,a,b)=-(1-V)\cos 2s$, $A_3(t_1,t_2)=-\frac{V}{\tilde r}\sin\tilde\phi$ and $A_4(t_1,t_2)=\frac{V}{\tilde r}\cos\tilde\phi$. Outside $\Sigma_{\epsilon}$, $(\psi_0, A_0)=(We^{i\varphi}, Zd\varphi)$, where $W,Z=1+O(e^{-\delta r})$ and $\varphi$ is an extension of the angular function $\phi$. 

\begin{remark}
    We remark here that our approximate solution is slightly different from that of \cite{LMWW}. We add $(1-V)d\theta$ so that $B_0$ is of $O(\epsilon^2)$. The existence of such a solution is ensured with similar arguments as that of \cite{LMWW}.
\end{remark}

To analyze the stability of such a solution, we need to study the linearized operator at $(\psi,A)$, defined by
\begin{equation*}
L(\xi,B,\psi,A):=
\begin{pmatrix}
-\Delta_A \xi-id^*B\psi +2i\langle B,\nabla_A\psi\rangle+\frac{\lambda}{2}(\psi^2\overline{\xi}+2|\psi|^2\xi-\xi)\\
d^*d B+Im(\overline{\nabla_A\psi}\cdot \xi+\overline{\nabla_A\xi}\cdot \psi)+B|\psi|^2
    \end{pmatrix}.
\end{equation*}

For any $\gamma\in C^1(\R^4)$, we define the Gauge transformations
\begin{equation*}
G_{\gamma}(\xi,B):=(e^{i\gamma}\xi, B+d\gamma)\text{ and }\tilde G_{\gamma}(\xi,B):=(e^{i\gamma}\xi, B).
\end{equation*}

\begin{definition}
$L(\cdot; \psi, A)$ is called to be stable if for any $v=(\xi, B)$(integrable in suitable sense),
\begin{equation*}
\delta^2 \mathcal E[v,v]=\int_{\R^4} \langle L(v;\psi, A), v\rangle \geq 0.
\end{equation*}
\end{definition}

$L(\cdot;\psi,A)$ is invariant under the following Gauge transformation
\begin{equation*}
L(\tilde G_{\gamma}(\xi,B);G_{\gamma}(\psi,A))=\tilde G_{\gamma}(L(\xi,B)).
\end{equation*}
Consequently, $L(\cdot;\psi, A)$ admits an infinite-dimensional subspace of bounded kernels
\begin{equation*}
  (i\gamma \psi,d\gamma)\text{ for any } \gamma\in C^1(\R^4).
\end{equation*}
Also, $L(\cdot;\psi, A)$ admits several bounded kernels as following:
\begin{equation*}
\begin{aligned}
\tilde Z_j:=(\partial_j \psi, \partial_j A),\ j=1,2,3,4,\ \tilde Z_5:=(x\cdot \nabla\psi,x\cdot \nabla A+A),\ \tilde Z_6:=(Jx\cdot \nabla \psi, Jx\cdot \nabla A+JA),      
\end{aligned}   
\end{equation*}
where $\nabla A:=\nabla A^j dx_j$ for $A=A^j dx_j$. Here $\tilde Z_j$, $j=1,2,3,4$ are generated by translations in $x_j$-direction. $\tilde Z_5$ is generated by dilation
\begin{equation*}
(\psi(x), A^j(x)d x_j)\mapsto (\psi(tx), tA^{j}(tx)dx_j).
\end{equation*}
Finally, $\tilde Z_6$ is generated by the an action in $O(4)/SU(2)$
\begin{equation*}
(\psi(x), A^j(x)d x_j)\mapsto (\psi(Jx), A^{j}(Jx)d((Jx)_j)).
\end{equation*}

Even though these $\tilde Z_j$'s are not in $L^2(\R^4)$, we can use the gauge kernels $(i\gamma \psi, d\gamma)$ to modify them so that they decay exponentially. Specifically, the modified kernels are
\begin{equation}\label{modifiedkernel}
\begin{aligned}
Z_j:=\tilde Z_j-(iA^j\psi, dA^j)=(\partial_j\psi-iA^j\psi, \partial_jA-dA^j),\ j=1,2,3,4,\\
Z_5:=\tilde Z_5-(i(x\cdot A)\psi, d(x\cdot A))=(x\cdot \nabla_A\psi, x_j(\partial_j A-dA^j)),\\
Z_6:=\tilde Z_6-(i(Jx\cdot A)\psi, d(Jx\cdot A))=(Jx\cdot \nabla_{A}\psi, (Jx)_j(\partial_j-dA^j)).
\end{aligned}
\end{equation}

We define $Z_{sym}$ to be the space  spanned by the $Z_j$ (or $\tilde Z_j$), $j=1,\dots, 6$ and the gauge kernels. We will show in Theorem \ref{main} that $L(\cdot;\psi,A)$ is non-degenerate in the sense that $Z_{sym}$ is actually the space of all bounded kernels of $L(\cdot;\psi,A)$. However, it is inconvenient to deal with the infinite-dimensional gauge kernels. Also, the operator $L(\cdot; \psi, A)$ is not uniformly elliptic. To solve these problems, we restrict the perturbation $v$ into the space that is orthogonal to all the gauge kernels. To see this, we notice that 
\begin{equation*}
    \int_{\R^4}\langle v, (i\gamma\psi, d\psi)\rangle=0\text{ for any }\gamma \in C^1(\R^4)
\end{equation*}
is equivalent to the gauge condition
\begin{equation*}
\Im (\overline{\psi}\xi)=d^*B.
\end{equation*}
Hence, we introduce the modified quadratic form
\begin{equation*}
\int_{\R^4}\langle \mathbb L(v;\psi,A), v\rangle:=\int_{\R^4}\langle  L(v;\psi,A), v\rangle+\int_{\R^4}(\Im (\overline{\psi}\xi)-d^*B)^2
\end{equation*}
and the (Gauge-fixed) linearized operator at $(\psi,A)$ as follows:
\begin{equation*}
    \mathbb L(\xi,B;\psi,A):=
    \begin{pmatrix}
        -\Delta_A \xi +2i\langle B,\nabla_A\psi\rangle+\frac12 (\lambda-1)\psi^2\bar\xi+(\lambda+\frac12)|\psi|^2\xi-\frac\lambda 2\xi\\
        \Delta_H B+2Im(\overline{\nabla_A\psi}\cdot \xi)+B|\psi|^2
    \end{pmatrix},
\end{equation*}
where $\Delta_H=d^*d+dd^*$ is the Hodge Laplacian. Now $\mathbb L(\cdot, \psi, A)$ is uniformly elliptic and the stability of $L(v;\psi, A)$ is established if we show that there exists a constant $c>0$ such that
\begin{equation*}
\int_{\R^4}\langle \mathbb L(v;\psi,A), v\rangle \geq 0, \text{ for any }v\in L^2(\R^4).
\end{equation*}
Furthermore, the non-degeneracy of $L(\cdot;\psi,A)$ is reduced to the non-degeneracy of $L(\cdot;\psi,A)$ in the sense that any bouned kernels of $\mathbb L(\cdot;\psi,A)$ is a linear combination of $Z_j$, $j=1,\dots,6$.

The following two orthogonal approximate kernels $\mathcal T_i:=(T_i,T_{B_i})$ of $\mathbb L(\xi,B;\psi,A)$ play a crucial role in the later discussions.
\begin{equation*}
\begin{aligned}
T_1(t_1,t_2)&=\partial_{t_1}\psi_0 -iA_3\psi_0=(U'\cos\tilde\phi-i\frac{U}{\tilde r}(1-V)\sin\tilde\phi)e^{i\tilde\phi},\\ T_2(t_1,t_2)&=\partial_{t_2}\psi_0-iA_4\psi_0=(U'\sin\tilde\phi+i\frac{U}{\tilde r}(1-V)\cos\tilde\phi)e^{i\tilde\phi},\\
T_{B_1}(t_1,t_2)&=\partial_{t_1} A_0-dA_3=\partial_{t_1} A_0 d\theta-(\partial_{t_2}A_3-\partial_{t_1}A_4)dt_2=\epsilon\cos 2s \cos \tilde\phi V' d\tth +\frac{V'}{\tilde r}dt_2,\\ 
T_{B_2}(t_1,t_2)&=\partial_{t_2} A_0-dA_4=\partial_{t_2} A_0 d\theta+(\partial_{t_2}A_3-\partial_{t_1}A_4)dt_1=\epsilon\cos 2s \sin \tilde\phi V' d\tth -  \frac{V'}{\tilde r}dt_1.
\end{aligned}
\end{equation*}
More precisely, they are actually the kernels of $\mathbb L(\xi, B;\psi_0, A_0)$, the linearized operator at $(\psi_0,A_0)$. Their norms can be computed:
\begin{lemma}We have the following identities:
\begin{enumerate}
\item $$\int_{\R^2}|T_1|^2=\int_{\R^2}|T_2|^2=\pi\int_{0}^{+\infty}[\tilde r(U')^2+\frac{1}{\tilde r} U^2(1-V)^2]d\tilde r,$$
\item $$\Re\int_{\R^2}\overline{T_1}T_2=\Re\int_{\R^2}\overline{T_2}T_1=0,\ \Im\int_{\R^2}\overline{T_1}T_2=2\pi\int_{0}^{+\infty}UU'(1-V)d\tilde r,$$
\item $$\int_{\R^2}|T_{B_1}|^2=\int_{\R^2}|T_{B_2}|^2=2\pi\int_{0}^{+\infty}\frac{(V')^2}{\tilde r}d\tilde r+O(\epsilon^3),\ \int_{\R^2}\langle T_{B_1}, T_{B_2}\rangle=O(\epsilon^3).$$
\end{enumerate}
\end{lemma}

\subsection{Eigenvalue problems and apriori estimates}
In this subsection, we consider the eigenvalue problem of the Jacobi operator associated to nonpositive eigenvalues and derive the corresponding apriori estimates. We first introduce some weighted Sobolev norms on $\Gamma$ and $\Gamma_{\epsilon}\times \R^2$:
\begin{gather}\label{pbetanorm}
\|f\|_{p,\beta}:=\sup_{P\in \Gamma}\left(\int_{\Gamma}|f(P)|^p \rho(P)^{\beta} dvol_{\Gamma}\right)^{\frac 1p},\\
\|g\|_{0,\beta,p,\sigma}:=\sup_{(P,a,b)\in \Gamma_{\epsilon}\times \R^2}\rho(P)^{\beta}e^{\sigma r}\|g\|_{L^p(B_1(P,a,b))}\nonumber,\\
\|g\|_{2,\beta,p,\sigma}:=\|D^2g\|_{0,\beta+2,p,\sigma}+\|Dg\|_{0,\beta+1,p,\sigma}+\|g\|_{0,\beta,p,\sigma}\nonumber,\\
\|(\xi, B)\|_{0,\beta,p,\sigma}:=\|\xi\|_{0,\beta,p,\sigma}+\|B\|_{0,\beta,p,\sigma},\ \|(\xi, B)\|_{2,\beta,p,\sigma}:=\|\xi\|_{2,\beta,p,\sigma}+\|B\|_{2,\beta,p,\sigma}\nonumber 
\end{gather}
where $f(P)$ is a function on $\Gamma$ and $g(P,a,b)$ is a function or one-form, $\xi(P,a,b)$ is a function, $B(P,a,b)$ is an one-form on $\Gamma_{\epsilon}\times \R^2$.

The eigenvalue problem of the Jacobi operator $L_\Gamma$ arises from the second variation of the area functional
\begin{equation*}
\mathrm Q(N,N):=\int_{\Gamma}|\nabla_{\Gamma}^{\nu}N|-2\rho^{-6}|N|^2  dvol_{\Gamma},
\end{equation*}
where $N=k_1(s,\theta)\m+k_2(s,\theta)\n$ is a normal vector field on $\Gamma$ and $\nabla_{\Gamma}^{\nu}$ is the covariant derivative on the normal bundle $\mathcal{N}\Gamma$. It can be written explicitly as
\begin{equation*}
\begin{aligned}
\mathrm Q(N,N)=&\int_{\Gamma}\rho^{-6}(|k_{1s}|^2+|k_{2s}|^2)+\rho^{-2}(|k_{1\theta}|^2+|k_{2\theta}|^2)-2\rho^{-2}\cos 2s (k_{1}k_{2\theta}-k_{2}k_{1\theta})\\
&-\rho^{-2}(2\rho^{-4}-\cos^2 2s)(|k_1|^2+|k_2|^2)  dvol_{\Gamma},     
\end{aligned}
\end{equation*}

The related eigenvalue problem is
\begin{equation*}
\Delta_{\Gamma}^{\nu}N+2\rho^{-6}N+\mu \rho^{-6}N=0 \text{ in }\Gamma.
\end{equation*}

We also consider the region
\begin{equation*}
\Gamma^R:=\{(s,\theta)\in \Gamma:\ \rho(s,\theta)<R\}
\end{equation*}
for a large number $R$ and the eigenvalue problem in $\Gamma^R$
\begin{equation}\label{EigenProbR}
\begin{cases}
\Delta_{\Gamma}^{\nu}N+2\rho^{-6}N+\mu \rho^{-6}N=F &\text{ in }\Gamma^R,\\
N=0 &\text{ on }\partial\Gamma^R.
\end{cases}
\end{equation}
Then we have the following apriori estimates
\begin{lemma}\label{JacobiEigenEst}
Let $p>1$, $\sigma>0$. Then for any fixed $R_0>0$ large and $\mu_0>0$, there exists a constant $C>0$ such that for all $R>R_0+1$, $-\mu_0< \mu\leq 0$, normal vector field $F$, if \eqref{EigenProbR} admits a solution $N=k_1 \m+k_2 \n$, then for $i=1,2$, If $\|F\|_{p,\beta+2}<+\infty$, then
 \begin{equation}\label{EigenLinftyest}
\|N\|_{L^\infty}\leq C[\|F\|_{p,\beta+2}+\|N\|_{L^{\infty}(|\rho|<3R_0)}]
\end{equation}
and
\begin{equation}\label{Eigenw2pest}
\|\nabla_\Gamma^2 N\|_{p,\beta+2}+\|\nabla_\Gamma N\|_{p,\beta+1}+\|N\|_{p,\beta}\leq C[\|F\|_{p,\beta+2}+\|N\|_{L^{\infty}(|\rho|<3R_0)}].
\end{equation}
\end{lemma}
\begin{proof}
Since $-\mu_0<\mu\leq 0$, the proof \eqref{EigenLinftyest} and \eqref{Eigenw2pest} is almost the same as the proof of Lemma 10.2 in \cite{PW} and Proposition 8.2 in Liu-Ma-Wei-Wu \cite{LMWW} respectively.
\end{proof}

We also consider the following eigenvalue problem for $\mathbb L(\cdot; \psi, A)$
\begin{equation}\label{GLEigenBR}
\begin{cases}
\mathbb L(v;\psi,A)-\mu \rho^{-6}v=0&\text{ in }B_{\epsilon^{-1}R}\\
v=0&\text{ on }\partial B_{\epsilon^{-1}R}.
\end{cases}
\end{equation}
and the eigenvalue problem in $\R^4$
\begin{equation}\label{GLEigenR4}
\mathbb L(v;\psi,A)-\mu \rho^{-6}v=0\text{ in }\R^4.    
\end{equation}
For the eigenvalue problems, we will show that the corresponding eigenfunction decays exponentially away from $\Ge$.

\begin{lemma}\label{EigenDecay}
Let $v$ be a solution to \eqref{GLEigenBR} or \eqref{GLEigenR4} with $\mu\leq 0$. Then for any fixed $0<\delta<\frac{\sqrt2}{2}\min\{\sqrt \lambda, 1\}$, there exists a constant $C$ depends on $\mu$ and $\delta$ but independent of sufficiently small $\epsilon>0$ and large $R>0$, such that
\begin{equation*}
|v(\ts,\tth,t_1,t_2)|\leq C\|v\|_{L^\infty}e^{-\delta \tilde r}\text{ in }\Sigma_{\epsilon},
\end{equation*}
Outside $\Sigma_{\epsilon}$, we have
\begin{equation}\label{EigenExpDecayOutside}
|v(\ts,\tth,t_1,t_2)|<Ce^{-\delta r_\epsilon}\text{ in }\Sigma_{\epsilon}^c.
\end{equation}
\end{lemma}
\begin{proof}
We first assume that $v=(\xi, B)$ is a solution to \eqref{GLEigenBR}. Note that if we test it against $v$, we find that
\begin{equation*}
\begin{aligned}
0=&\langle \mathbb L(v;\psi,A)-\mu\rho^{-6}v, v\rangle=\Re( \mathbb L(v;\psi,A)_1\overline{\xi})+\langle\mathbb L(v;\psi,A)_2, B\rangle-\mu\rho^{-6}|v|^2 \\
=&-\frac12\Delta |v|^2+|\nabla v|^2-2\Im[\langle A,d\xi \rangle \overline{\xi}]+|A|^2|\xi|^2-4\Im[\langle B,\nabla_{A}\psi\rangle \overline{\xi}]\\
&+\frac{\lambda-1}{2}\Re(\psi^2\overline{\xi}^2)+(\lambda+\frac12)|\psi|^2|\xi|^2-\frac{\lambda}{2}|\xi|^2+|B|^2|\psi|^2-\mu\rho^{-6}|v|^2.
\end{aligned}
\end{equation*}
Note that $|A|=O(\tilde r^{-1}e^{-\tilde r})$, $1-|\psi|^2=O(e^{-m_{\lambda}\tilde r})$ and $|\nabla_{A}\psi|=O(e^{-\min\{1,\sqrt\lambda\}\tilde r})$ as $\tilde r\to \infty$. In the region $r_0<\tilde r<r_\epsilon$ for some sufficiently large constant $r_0>0$ depending on $\lambda$ and $\delta$, we have
\begin{equation*}
0\geq -\Delta|v|^2+4\delta^2 |v|^2-2\mu\rho^{-6}|v|^2\geq -\Delta|v|^2+4\delta^2 |v|^2.
\end{equation*}

Note that $w=e^{-2\delta \tilde r}+e^{-2(2r_\epsilon-\delta)\tilde r}$ is a positive supersolution that satisfies
\begin{equation*}
-\Delta|w|^2+4\delta^2 |w|^2>0\text{ in }r_0<r<r_\epsilon
\end{equation*}
if $\epsilon$ is chosen sufficiently small and $R>0$ is sufficiently large.
Then by maximum principle, we have
\begin{equation*}
|v|\leq C\|v\|_{L^{\infty}}e^{-\delta \tilde r}\text{ for }r_0<r<r_\epsilon.
\end{equation*}
Then \eqref{EigenExpDecayOutside} follows from the maximum principle. Then we proved the lemma for a solution to \eqref{GLEigenBR}. If $v$ is a solution to \eqref{GLEigenR4}, the proof is similar.
\end{proof}

We also list below the linear theory of $\mathbb L(\cdot; \psi_0, A_0)$ studied in Liu-Ma-Wei-Wu \cite{LMWW}.
\begin{lemma}[Theorem 6.3 in \cite{LMWW}]\label{lineartheory}
If $\|(\mu,C)\|_{0,2,p,\delta}<+\infty$ for some $p>1$ and $\delta>0$. Then for $\epsilon>0$ sufficiently small, there exists a unique solution $(\xi,B)$ to
\begin{equation*}
\begin{aligned}
\mathbb L(\xi,B;\psi_0,A_0)=(\mu,C)+c_1(s,\theta)\mathcal{T}_1+c_2(s,\theta)\mathcal{T}_2\ &\text{ for all }(s,\theta,t_1,t_2)\in \Gamma_\epsilon\times \R^2,\\
\int_{\R^2} \langle (\xi, B)(s,\theta,t_1,t_2), \mathcal{T}_j(t_1,t_2)\rangle dt_1dt_2=0\ &\text{  for all }(s,\theta)\in \Ge
\end{aligned}
\end{equation*}
with
\begin{equation*}
\|(\xi,B)\|_{2,2,p,\delta}\leq C_0 \|(\mu,C)\|_{0,2,p,\delta}
\end{equation*}
for some $C_0>0$ independent of $\epsilon$ and $(\mu, C)$. Here $c_j(s,\theta)$, $j=1,2$, is given by
\begin{equation*}
c_j(s,\theta)=-\frac{\int_{\R^2} \langle (\mu,C)(s,\theta,t_1,t_2), \mathcal{T}_j(t_1,t_2)\rangle dt_1 dt_2}{\int_{\R^2}|\mathcal{T}_j|^2}.  \end{equation*}
\end{lemma}

\subsection{Improvement of approximation}
Establishing the stability amounts to analyze the spectrum of the linearized operator around the solution. In our situation, it will be necessary to find a more accurate approximate solution. That is, a higher order expansion of $(\eta_0, B_0)$. 

Let 
\begin{equation*}
    S(\psi_0, A_0):=
    \begin{pmatrix}
        -\Delta_{A_0} \psi_0 +\frac{\lambda}{2}(|\psi_0|^2-1)\psi_0\\
        d^*dA_0-Im(\nabla_{A_0}\psi_0\cdot\bar\psi_0)
    \end{pmatrix}.
\end{equation*}
be the error of the approximate solution. Then direct computation tells us that
\begin{equation*}
\begin{aligned}
-\Delta_{A_0} \psi_0 +\frac{\lambda}{2}(|\psi_0|^2-1)\psi_0
&=2\epsilon^2\rho^{-6} \tilde rU'e^{i\tilde\phi}+O(\epsilon^3\rho^{-4}e^{-\tilde r})
\end{aligned}
\end{equation*}
and
\begin{equation*}
\begin{aligned}
d^*dA_0-Im(\nabla_{A_0}\psi_0\cdot\bar\psi_0)
&=O(\epsilon^3\rho^{-2}e^{-\tilde r})d\ts+O(\epsilon^3\rho^{-2}e^{-\tilde r})d\tth+[2\epsilon^2\rho^{-6} t_2(\partial_{t_2}A_3-\partial_{t_1}A_4)+O(\epsilon^3\rho^{-4}e^{-\tilde r})]dt_1\\
&-[2\epsilon^2\rho^{-6} t_1(\partial_{t_2}A_3-\partial_{t_1}A_4)+O(\epsilon^3\rho^{-4}e^{-\tilde r}))]dt_2\\
&=O(\epsilon^3\rho^{-2}e^{-\tilde r}))d\ts+O(\epsilon^3\rho^{-2}e^{-\tilde r})d\tth-[2\epsilon^2\rho^{-6} \sin\tilde\phi V'+O(\epsilon^3\rho^{-4}e^{-\tilde r}))]dt_1\\
&+[2\epsilon^2\rho^{-6}\cos\tilde\phi V'+O(\epsilon^3\rho^{-4}e^{-\tilde r}))]dt_2.
\end{aligned}
\end{equation*}
See also Lemma 4.1 to Lemma 4.5 in \cite{LMWW} for similar computation.

The $O(\epsilon^2)$-terms above are orthogonal to the approximate kernels $\mathcal T_1$ and $\mathcal T_2$ up to $O(\epsilon^2)$. However, they are still large terms. In order to cancel these terms, we introduce the following improved approximate solution. Let $(\eta_1, B_1)$ be the solution to the equation
\begin{equation*}
    \mathbb L(\eta_1,B_1;\psi^{(1)},A^{(1)})=
    \begin{pmatrix}
        \tilde rU'e^{i\tilde\phi}\\
        -\sin\tilde\phi V' dt_1+\cos\tilde\phi V' dt_2
    \end{pmatrix}
    \text{ in }\R^2.
\end{equation*}
Note that the right-hand side is perpendicular to $\widetilde{\mathcal T}_1:=\mathcal T_1-(0,\cos 2s \sin \tilde\phi V' d\theta)$ and $\widetilde{\mathcal T}_2:=\mathcal T_2-(0,\cos 2s \sin \tilde\phi V' d\theta)$, the kernels of $\mathbb L(\eta_1,B_1;\psi^{(1)},A^{(1)})$ in $\R^2$. Thus the existence of such $(\eta_1,B_1)$ is ensured. Then the similar argument as in \cite{LMWW} implies that the solution constructed in \cite{LMWW} can be written as
\begin{equation*}
\begin{aligned}
(\psi,A)=(\psi_0,A_0)+(\eta_0,B_0)
=
(\psi_0,A_0)-2\epsilon^2\rho^{-6}
(\eta_1,B_1)
+O(\epsilon^3\rho^{-4}e^{-\delta \tilde r}).
\end{aligned}
\end{equation*}

\section{Analysis of the linearized operator}
The idea of proof of stability is as follows: If $(\xi,B)$ is an eigenfunction of $\mathbb L(\cdot;\psi,A)$ w.r.t. some negative eigenvalue, then it can be written as $(\xi,B)\approx k_1(\ts,\tth)\mathcal{T}_1(t_1,t_2)+k_2(\ts,\tth)\mathcal{T}_2(t_1,t_2)$, where $N=k_1(\ts,\tth)\m+k_2(\ts,\tth)\n$ is a negative direction of the Jacobi operator $L_\Gamma$. Then the stability of $L_{\Gamma}$ can be applied. As a consequence, a detailed analysis of the linearized operator $\mathbb L(\cdot;\psi,A)$ is required. To start with, we need the following proposition to build up the relationship between $\mathbb L(\cdot; \psi, A)$ and  $L_\Gamma$.

\begin{proposition}\label{TestKernel}
Let $N=k_1(\ts,\tth)\m+k_2(\ts,\tth)\n$ be a normal vector field. Consider the vector-valued function defined for $x\in \Sigma_\epsilon$ by
\begin{equation*}
v(x)=v(\ts,\tth,t_1,t_2)=(\xi,B)(\ts,\tth,t_1,t_2)=k_1(\ts,\tth)\mathcal T_1(t_1,t_2)+k_2(\ts,\tth)\mathcal T_2(t_1,t_2).
\end{equation*}
Then we have
\begin{equation*}
\begin{aligned}
    &\int_{(a,b)\in \Sigma_\epsilon}\langle \mathbb L(v;\psi,A), \mathcal T_1 \rangle dt_1dt_2\m+\int_{(a,b)\in \Sigma_\epsilon}\langle \mathbb L(v;\psi,A), \mathcal T_2\rangle dt_1dt_2\n\\
    =&-L_{\Ge}N\int_{\R^2}|\mathcal T_1|^2+O(\epsilon^3\rho^{-4}(\partial_{ij}N+\partial_{i}N+N)).
\end{aligned}
\end{equation*}
\end{proposition}

To prove the proposition above, we calculate the integrals term by term. For simplicity, we only compute the first integral. The computation of the second integral is similar. The following lemma computes the inner product of the first components of $\mathbb L(v;\psi, A)$ and $\mathcal{T}_1$.
\begin{lemma}\label{InnerProduct1}
\begin{equation*}
\begin{aligned}
&\Re\int_{(a,b)\in \Sigma_\epsilon}(\mathbb L(v;\psi,A))_1\cdot \overline{T_1}\\
=&-[\Delta_{\Gamma_{\epsilon}}k_1-2\epsilon\rho^{-2}\cos 2sk_{2\tth}+\epsilon^2\rho^{-2}(2\rho^{-4}-\cos^2 2\epsilon \ts)k_1+O(\epsilon^2\nabla^2_{\Ge}k_1)]\int_{\R^2}|T_1|^2\\
+&k_2\Im\int_{\R^2}d^*B_0 T_2\overline{T_1}-2\Im\int_{\R^2}\langle B_0,d\xi\rangle \overline{T_1}+\Re\int_{\R^2}(|A^{(1)}+B_0|^2-|A^{(1)}|^2)(k_1|T_1|^2+k_2T_2\overline{T_1})\\
+&O(\epsilon^3\rho^{-4}(\partial_{ij}k_l+\partial_{i}k_l+k_l)).
\end{aligned}
\end{equation*}
\end{lemma}
\begin{proof}
From the definition of $\mathbb L$, we see that
\begin{equation*}
(\mathbb L(v;\psi,A))_1=-\Delta_A \xi +2i\langle B,\nabla_A\psi\rangle+\frac12 (\lambda-1)\psi^2\bar\xi+(\lambda+\frac12)|\psi|^2\xi-\frac\lambda 2\xi.
\end{equation*}
In the following, we test $T_1$ and calculate the integrals one by one.

\textit{Step 1: The integral of $-\Delta_A\xi$.}

From Lemma \ref{Laplace1Expansion}, we have
\begin{equation*}
\begin{aligned}
\Delta\xi
&=(T_{1t_1t_1}+T_{1t_2t_2})k_1+(T_{2t_1t_1}+T_{2t_2t_2})k_2+\epsilon^2\rho^{-2}\cos^2 (2\epsilon\ts) k_1(t_1^2T_{1t_2t_2}+t_2^2T_{1t_1t_1}-2t_1t_2T_{1t_1t_2})\\
&+\Delta_{\Gamma_{\epsilon}}k_1 T_1-\epsilon^2\rho^{-2}(\rho^{-4}+1)k_1(t_1T_{1t_1}+t_2T_{1t_2})+2\epsilon\rho^{-2}\cos (2\epsilon\ts) k_{1\tth}(t_2T_{1t_1}-t_1T_{1t_2})\\
&+\epsilon^2\rho^{-2}\cos^2 2\epsilon \ts k_2(t_1^2T_{2t_t2t_2}+t_2^2T_{2t_1t_1}-2t_1t_2T_{2t_1t_2})+\Delta_{\Gamma_{\epsilon}}k_2 T_2-\epsilon^2\rho^{-2}(\rho^{-4}+1)k_2(t_1T_{2t_1}+t_2T_{2t_2})\\
&+2\epsilon\rho^{-2}\cos (2\epsilon\ts) k_{2\tth}(t_2T_{2t_1}-t_1T_{2t_2})-4\epsilon^2 t_1(t_2\xi_{\tth t_1}-t_1\xi_{\tth t_2})\rho^{-5}\cos (2\epsilon\ts)\\
&+2\epsilon^2 \rho^{-2}\cos (2\epsilon\ts) (f_2\xi_{t_1}-f_1\xi_{t_2})+4t_2\epsilon^2\rho^{-7}\cos2\epsilon\ts(t_2\xi_{\ts t_1}-t_1\xi_{\ts t_2})-2\epsilon^2\rho^{-6}(f_{1s}\xi_{\ts t_1}+f_{2s}\xi_{\ts t_2})\\
&-2\epsilon^2\rho^{-2}(f_{1\theta}\xi_{\tth t_1}+f_{2\theta}\xi_{\tth t_2})+(2(t_1+\epsilon f_1)\rho^{-9}+3\tilde r^2\rho^{-12}\epsilon)\epsilon\xi_{\ts\ts}+4(t_2+\epsilon f_2)\epsilon\rho^{-7}\xi_{\ts\tth}\\
&+(-2(t_2+\epsilon f_2)\rho^{-5}+3\tilde r^2\epsilon \rho^{-8})\epsilon\xi_{\tth\tth}+8t_1\epsilon^2\rho^{-7}\cos^2 (2\epsilon\ts) \xi_{\ts}+4t_2\epsilon^2\rho^{-5}\cos (2\epsilon\ts) \xi_{\tth}+O(\epsilon^3\rho^{-4}).
\end{aligned}
\end{equation*}

Denote
\begin{equation*}
\begin{aligned}
Q_1:=&-4\epsilon^2 t_1(t_2\xi_{\tth t_1}-t_1\xi_{\tth t_2})\rho^{-5}\cos (2\epsilon\ts)+2\epsilon^2 \rho^{-2}\cos (2\epsilon\ts) (f_2\xi_{t_1}-f_1\xi_{t_2})\\
&+4t_2\epsilon^2\rho^{-7}\cos2\epsilon\ts(t_2\xi_{\ts t_1}-t_1\xi_{\ts t_2})-2\epsilon^2\rho^{-6}(f_{1s}\xi_{\ts t_1}+f_{2s}\xi_{\ts t_2})-2\epsilon^2\rho^{-2}(f_{1\theta}\xi_{\tth t_1}+f_{2\theta}\xi_{\tth t_2})\\
&+2t_1\epsilon\rho^{-9}\xi_{\ts\ts}+4t_2\epsilon\rho^{-7}\xi_{\ts\tth}-2t_2\epsilon\rho^{-5}\xi_{\tth\tth}+8t_1\epsilon^2\rho^{-7}\cos^2 (2\epsilon\ts) \xi_{\ts}+4t_2\epsilon^2\rho^{-5}\cos (2\epsilon\ts) \xi_{\tth}.
\end{aligned}
\end{equation*}
We see that $Q_1$ is orthogonal to both $T_1$ and $T_2$. Then we test $T_1(t_1,t_2)$ and integrate by parts in $(t_1,t_2)$. Note that
\begin{equation*}
\begin{aligned}
&\Re\int_{\R^2}(t_1T_{1t_1}+t_2T_{1t_2})\overline{T_1}=\frac12\int_{\R^2}(t_1,t_2)\cdot \nabla |T_1|^2=-\int_{\R^2}|T_1|^2,
\end{aligned}
\end{equation*}

\begin{equation*}
\begin{aligned}
&\Re\int_{\R^2}(t_2^2T_{1t_2t_2}+t_1^2T_{1t_1t_1}-2t_1t_2T_{1t_1t_2})\overline{T_1}=\Re\int_{\R^2}(t_2\partial_{t_1}-t_1\partial_{t_2})^2T_1\cdot\overline{T_1}+(t_1\partial_{t_1}+t_2\partial_{t_2})T_1\cdot \overline{T_1}\\
=&\Re\int_{\R^2}\partial_{\tilde\phi}^2T_1\cdot \overline{T_1}-|T_1|^2=-3\pi\int_{0}^{+\infty}\tilde r(U')^2+\frac1r U^2(1-V)^2 dr+4\pi \int_{0}^{+\infty}UU'(1-V)dr,
\end{aligned}
\end{equation*}

\begin{equation*}
\begin{aligned}
&\Re\int_{\R^2}(t_2T_{1t_1}-t_1T_{1t_2})\overline{T_1}=\frac12\int_{\R^2}(t_2,-t_1)\cdot \nabla |T_1|^2=0,
\end{aligned}
\end{equation*}

\begin{equation*}
\begin{aligned}
&\Re\int_{\R^2}(t_1T_{2t_1}+t_2T_{2t_2})\overline{T_1}=\Re\int_{\R^2}\tilde r\partial_{\tilde r}T_2\cdot \overline{T_1}=0,
\end{aligned}
\end{equation*}

\begin{equation*}
\begin{aligned}
&\Re\int_{\R^2}(t_2^2T_{2t_2t_2}+t_1^2T_{2t_1t_1}-2t_1t_2T_{2t_1t_2})\overline{T_1}=\Re\int_{\R^2}(t_2\partial_{t_1}-t_1\partial_{t_2})^2T_2\cdot\overline{T_1}+(t_1\partial_{t_1}+t_2\partial_{t_2})T_2\cdot \overline{T_1}\\
=&\Re\int_{\R^2}\partial_{\tilde\phi}^2T_2\cdot \overline{T_1}=0,
\end{aligned}
\end{equation*}

\begin{equation*}
\begin{aligned}
&\Re\int_{\R^2}(t_2T_{2t_1}-t_1T_{2t_2})\overline{T_1}=-\Re\int_{\R^2}\partial_{\tilde\phi}T_2\cdot \overline{T_1}=2\pi\int_{0}^{+\infty}[\tilde r(U')^2+\frac{1}{\tilde r}U^2(1-V)^2-2UU'(1-V)]d\tilde r.
\end{aligned}
\end{equation*}

Combining the integrals above, we obtain
\begin{equation}\label{Int-Deltaxi}
\begin{aligned}
\Re\int_{(a,b)\in \Sigma_\epsilon}-\Delta\xi\cdot \overline{T_1}
=&\Re \int_{\R^2}-(T_{1t_1t_1}+T_{1t_2t_2})k_1\overline{T_1}-(T_{2t_1t_1}+T_{2t_2t_2})k_2\overline{T_1}\\
-&[\Delta_{\Gamma_\epsilon}k_1-2\epsilon\rho^{-2}\cos (2\epsilon\ts) k_{2\tth}+2\epsilon^2  \rho^{-2}(\rho^{-4}-\cos^2 (2\epsilon\ts))k_1+O(\epsilon^2\nabla^2_{\Ge}k_1)]\int_{\R^2}|T_1|^2\\
-&4(\epsilon\rho^{-2}\cos (2\epsilon\ts) k_{2\tth}+\epsilon^2\rho^{-2}\cos^2 (2\epsilon\ts) k_1)\pi\int_{0}^{+\infty}UU'(1-V)d\tilde r+O(\epsilon^3\rho^{-4}).
\end{aligned}
\end{equation}

By \eqref{Laplace2Expansion}, we have
\begin{equation*}
d^*A\xi=d^*B_0(k_1T_1+k_2T_2)+O(\epsilon^3).
\end{equation*}
Likewise, we test $T_1(t_1,t_2)$ and integrate by parts in $(t_1,t_2)$ to get
\begin{equation}\label{IntdstrAxi}
\begin{aligned}
    \Re\int_{\R^2}-id^*A\xi\overline{T_1}=\Im\int_{\R^2}d^*B_0(k_1T_1+k_2T_2)\overline{T_1}+O(\epsilon^3\rho^{-4})=k_2\Im\int_{\R^2}d^*B_0 T_2\overline{T_1}+O(\epsilon^3\rho^{-4}).
\end{aligned}
\end{equation}

By \eqref{Laplace3Expansion}, there holds
\begin{equation*}
\begin{aligned}
\langle A,d\xi\rangle&=A_3\xi_a+A_4\xi_b-\epsilon^2\rho^{-2}\cos^2 (2\epsilon\ts) (b\xi_a-a\xi_b)-\epsilon \rho^{-2} \cos (2\epsilon\ts) \xi_{\tth}\\
&-\epsilon^2\rho^{-6}(2t_2\rho^{-1}\cos (2\epsilon\ts)+A_3 f_{1s}+A_4 f_{2s})\xi_{\ts}+\epsilon^2\rho^{-2}(2t_1\rho^{-1}\cos (2\epsilon\ts)-A_3 f_{1s}-A_4 f_{2s})\xi_{\tth}\\
&+\langle B_0,d\xi\rangle+O(\epsilon^3\rho^{-4})\\
&=:A_3\xi_a+A_4\xi_b-\epsilon^2\rho^{-2}\cos^2 (2\epsilon\ts) (b\xi_a-a\xi_b)-\epsilon \rho^{-2} \cos (2\epsilon\ts) \xi_{\tth}+Q_2+\langle B_0,d\xi\rangle+O(\epsilon^3\rho^{-4}).    
\end{aligned}
\end{equation*}
We find that $Q_2$ is orthogonal to both $T_1$ and $T_2$. Then we test $T_1(t_1,t_2)$ and integrate by parts in $(t_1,t_2)$ and deduce that
\begin{equation}\label{IntAdxi}
\begin{aligned}
\Re\int_{\R^2}2i\langle A,d\xi\rangle\overline{T_1}=&\Re\int_{\R^2}2i(A_3T_{1t_1}+A_4T_{1t_2})k_1\overline{T_1}+2i(A_3T_{2t_1}+A_4T_{2t_2})k_2\overline{T_1}\\
&-2\epsilon^2\rho^{-2}\cos^2 (2\epsilon\ts) k_1 \Im\int_{\R^2}(t_2T_{1t_1}-t_1T_{1t_2})\overline{T_1}\\
&-2\epsilon^2\rho^{-2}\cos^2 (2\epsilon\ts) k_2 \Im\int_{\R^2}(t_2T_{2t_1}-t_1T_{2t_2})\overline{T_1}\\
&-2\epsilon\rho^{-2}\cos (2\epsilon\ts) k_{2\tth}\Im\int_{\R^2}T_2\overline{T_1}-2\Im\int_{\R^2}\langle B_0,d\xi\rangle \overline{T_1}+O(\epsilon^3)\\
=&\Re\int_{\R^2}2i(A_3T_{1t_1}+A_4T_{1t_2})k_1\overline{T_1}+2i(A_3T_{2t_1}+A_4T_{2t_2})k_2\overline{T_1}\\
&-2\epsilon^2\rho^{-2}\cos^2 (2\epsilon\ts) k_1\pi\int_{0}^{+\infty}\tilde r(U')^2+\frac{1}{\tilde r} U^2(1-V)^2-2UU'(1-V) d\tilde r\\
&+4\epsilon\rho^{-2} \cos (2\epsilon\ts) k_{2\tth}\pi\int_{0}^{+\infty} UU'(1-V)d\tilde r-2\Im\int_{\R^2}\langle B_0,d\xi\rangle \overline{T_1}+O(\epsilon^3).
\end{aligned}
\end{equation}

Applying \eqref{Laplace4Expansion}, we have
\begin{equation*}
|A|^2\xi=(A_3^2+A_4^2)(k_1T_1+k_2T_2)+\epsilon^2\rho^{-2}\cos^2 (2\epsilon\ts)(k_1T_1+k_2T_2)+(|A_0+B_0|^2-|A_0|^2)\xi+O(\epsilon^3).
\end{equation*}
Then we test $T_1(t_1,t_2)$ and integrate by parts in $(t_1,t_2)$:
\begin{equation}\label{IntA2xi}
\begin{aligned}
\Re\int_{\R^2}|A|^2\xi \overline{T_1}
=&\Re\int_{\R^2}(A_3^2+A_4^2)(k_1|T_1|^2+k_2T_2\overline{T_1})+\epsilon^2\rho^{-2}\cos^2 (2\epsilon\ts)\Re\int_{\R^2}(k_1|T_1|^2+k_2T_2\overline{T_1})\\
&+\Re\int_{\R^2}(|A_0+B_0|^2-|A_0|^2)(k_1|T_1|^2+k_2T_2\overline{T_1})+O(\epsilon^3)\\
=&k_1\Re\int_{\R^2}(A_3^2+A_4^2)|T_1|^2+\epsilon^2\rho^{-2}\cos^2 (2\epsilon\ts) \cdot \pi\int_{0}^{+\infty}\tilde r(U')^2+\frac{1}{\tilde r}U^2(1-V)^2 d\tilde r\\
&+k_1\Re\int_{\R^2}(|A_0+B_0|^2-|A_0|^2)|T_1|^2+O(\epsilon^2).
\end{aligned}
\end{equation}

Combining \eqref{Int-Deltaxi}, \eqref{IntdstrAxi}, \eqref{IntAdxi} and \eqref{IntA2xi}, we get
\begin{equation}\label{Int-DeltaAxi}
\begin{aligned}
&\Re\int_{(a,b)\in \Sigma_\epsilon}-\Delta_A\xi\cdot \overline{T_1}
=\Re \int_{\R^2}[-T_{1t_1t_1}-T_{1t_2t_2}+2i(A_3T_{1t_1}+A_4T_{1t_2})+(A_3^2+A_4^2)T_1]k_1\overline{T_1}\\
-&[\Delta_{\Gamma_{\epsilon}}k_1-2\epsilon\rho^{-2}\cos (2\epsilon\ts)k_{2\theta}+\epsilon^2\rho^{-2}(2\rho^{-4}-\cos^2 (2\epsilon \ts))k_1+O(\epsilon^2\nabla^2_{\Ge}k_1)]\int_{\R^2}|T_1|^2\\
+&k_2\Im\int_{\R^2}d^*B_0 T_2\overline{T_1}-2\Im\int_{\R^2}\langle B_0,d\xi\rangle \overline{T_1}+\Re\int_{\R^2}(|A^{(1)}+B_0|^2-|A^{(1)}|^2)(k_1|T_1|^2+k_2T_2\overline{T_1})+O(\epsilon^3).
\end{aligned}
\end{equation}

\textit{Step 2: The integral of $2i\langle B, \nabla_A\psi \rangle$.}

Note that $\nabla_{A}\psi=d\psi-iA\psi=\nabla_{A_0}\psi_0+\nabla_{A_0}\eta_0-iB_0\psi_0-B_0\eta_0$. Then we have
\begin{equation*}
\begin{aligned}
\nabla_{A_0}\psi_0=d\psi_0-iA_0\psi_0=&(-\epsilon^2 f_s T_1-\epsilon^2 g_s T_2)d\ts+(i(1-V)\epsilon\cos (2\epsilon\ts)\psi_0-\epsilon^2 f_\theta T_1-\epsilon^2 g_\theta T_2)d\tth\\
&+T_1 da+T_2 db.
\end{aligned}
\end{equation*}
From this we infer
\begin{equation*}
\begin{aligned}
\langle B,\nabla_{A_0}\psi_0\rangle
&=k_1\langle T_{B_1},\nabla_{A_0}\psi_0\rangle+k_2\langle T_{B_2},\nabla_{A_0}\psi_0\rangle=k_1\frac{V'}{\tilde r}T_2+k_2\frac{V'}{\tilde r}T_1+O(\epsilon^3\rho^{-4}e^{-\delta r}).
\end{aligned}
\end{equation*}
Then we test $T_1(t_1,t_2)$ and integrate in $(t_1,t_2)$ to get
\begin{equation}\label{IntBDApsi}
\begin{aligned}
\Re\int_{\R^2}2i\langle B,\nabla_{A_0}\psi_0\rangle\overline{T_1}
=&\Re\int_{\R^2}2ik_1\frac{V'}{\tilde r}T_2\overline{T_1}+2ik_2\frac{V'}{\tilde r}|T_1|^2+O(\epsilon^3\rho^{-4}).
\end{aligned}
\end{equation}

Combining the integrals above, we get
\begin{equation*}
\begin{aligned}
\Re\int_{\R^2}2i\langle B,\nabla_{A}\psi\rangle \overline{T_1}
=&\Re\int_{\R^2}2ik_1(\partial_{t_2} A_3-\partial_{t_1}A_4)|T_1|^2+\Re\int_{\R^2}2i\langle B, \nabla\eta_0-iA_0\eta_0-iB_0\psi_0-B_0\eta_0\rangle\overline{T_1}+O(\epsilon^3\rho^{-4}).
\end{aligned}
\end{equation*}

\textit{Step 3: Conclusion.}

Observe that $T_1$ satisfies
\begin{equation*}
-\Delta_{\R^2}T_1+2i A^{(1)}\cdot T_1+|A^{(1)}|^2T_1+2i(\partial_{t_2} A_3-\partial_{t_1}A_4)T_1+\frac12 (\lambda-1)(\psi^{(1)})^2\bar T_1+(\lambda+\frac12)|\psi^{(1)}|^2 T_1-\frac\lambda 2 T_1=0\text{ in }\R^2.
\end{equation*}
Combining it with \eqref{Int-DeltaAxi} and \eqref{IntBDApsi}, we obtain the desired result.
\end{proof}

The inner product of the second components of $\mathbb{L}(v;\psi,A)$ and $\mathcal{T}_2$ can also be computed as follows.

\begin{lemma}\label{InnerProduct2}
\begin{equation*}
\begin{aligned}
\int_{\R^2}\langle (\mathbb L(v;\psi,A))_2, T_{B_1}\rangle
=&-[\Delta_{\Ge}k_1-2\epsilon\rho^{-2}\cos 2sk_{2\tth}-\epsilon^2\rho^{-2}\cos^2 (2\epsilon\ts) k_1+O(\epsilon^2\nabla^2_{\Ge}k_1)]\int_{\R^2}|T_{B_1}|^2+O(\epsilon^3\rho^{-4})
\end{aligned}
\end{equation*}
\end{lemma}
\begin{proof}
Recall that
\begin{equation*}
\mathbb L(v;\psi,A))_2=\Delta_H B+2Im(\overline{\nabla_A\psi}\cdot \xi)+B|\psi|^2.
\end{equation*}
In the following, we will test $T_{B1}$ and compute the integrals term by term.

Firstly, by Lemma \ref{d*d}, we have
\begin{equation*}
\begin{aligned}
&\Delta_H B
=d^*dB+dd^*B=-\frac{1}{\sqrt{G}}g_{ml}\partial_j(\sqrt{G}g^{ij}g^{kl}B_{ik})dy^m-d\left(-\frac{1}{\sqrt{G}}\partial_j(\sqrt{G}g^{ij}B_i)\right)\\
=&O(\epsilon^2)d\ts+\left[\frac{2}{r^2}\epsilon(k_1\cos\tilde\phi+k_2\sin\tilde\phi)\cos (2\epsilon\ts) \left(\frac{V''}{\tilde r}-\frac{V'}{\tilde r}\right)+O(\epsilon^2)\right]d\tth+O(\epsilon^3)da+O(\epsilon^3)db\\
+&\left[(\Delta_{\Gamma_{\epsilon}}k_2-2\epsilon\rho^{-2}\cos (2\epsilon \ts) k_{1\tth})\frac{V'}{\tilde r}+k_2\left(\frac{V'''}{\tilde r}-\frac{V''}{\tilde r^2}+\frac{V'}{\tilde r}-\epsilon^2 \rho^{-2}\cos^2 (2\epsilon\ts) \frac{V'}{\tilde r}-2\epsilon^2 \rho^{-6} V''\right)+Q_3+Q_4\right]da\\
+&\left[(-\Delta_{\Gamma_{\epsilon}}k_1+2\epsilon\rho^{-2}\cos (2\epsilon\ts) k_{2\tth})\frac{V'}{\tilde r}-k_1\left(\frac{V'''}{\tilde r}-\frac{V''}{\tilde r^2}+\frac{V'}{\tilde r}-\epsilon^2 \rho^{-2}\cos^2 (2\epsilon\ts) \frac{V'}{\tilde r}-2\epsilon^2 \rho^{-6} V''\right)+Q_5+Q_6\right]db,
\end{aligned}
\end{equation*}
where
\begin{equation*}
\begin{aligned}
Q_3:=&2\epsilon^2 \rho^{-2}\cos (2\epsilon\ts) (f_2k_2\cos\tilde\phi-f_1k_2\sin\tilde\phi)\left(\frac{V'}{\tilde r}\right)^\prime-2\epsilon^2\rho^{-6}(f_{1s}k_{2\ts }\cos+f_{2s}k_{2\ts}\sin\tilde\phi)\left(\frac{V'}{\tilde r}\right)^\prime\\
&-2\epsilon^2\rho^{-2}(f_{1\theta}k_{2\tth}\cos\tilde\phi+f_{2\theta}k_{2\tth}\sin\tilde\phi)\left(\frac{V'}{r}\right)^\prime\\
&+(2t_1\epsilon\rho^{-9}k_{2\ts\ts}+4t_2\epsilon\rho^{-7}k_{2\ts\tth}-2t_2\epsilon\rho^{-5}k_{2\tth\tth}+8t_1\epsilon^2\rho^{-7}\cos^2 (2\epsilon\ts) k_{2\ts}+4t_2\epsilon^2\rho^{-5}\cos (2\epsilon\ts) k_{2\tth})\frac{V'}{\tilde r},
\end{aligned}
\end{equation*}
\begin{equation*}
\begin{aligned}
Q_4:=&[(2f_1\rho^{-9}+3\tilde r^2\rho^{-12}) k_{2\ts\ts}+4 f_2\rho^{-7}k_{2\ts\tth}+(-2 f_2\rho^{-5}+3\tilde r^2 \rho^{-8}) k_{2\tth\tth}]\epsilon^2\frac{V'}{\tilde r},
\end{aligned}
\end{equation*}
\begin{equation*}
\begin{aligned}
Q_5:=&-2\epsilon^2 \rho^{-2}\cos (2\epsilon\ts) (f_2k_1\cos\tilde\phi-f_1k_1\sin\tilde\phi)\left(\frac{V'}{\tilde r}\right)^\prime+2\epsilon^2\rho^{-6}(f_{1s}k_{1\ts }\cos+f_{2s}k_{1\ts}\sin\tilde\phi)\left(\frac{V'}{\tilde r}\right)^\prime\\
&+2\epsilon^2\rho^{-2}(f_{1\theta}k_{1\tth}\cos\tilde\phi+f_{2\theta}k_{1\tth}\sin\tilde\phi)\left(\frac{V'}{r}\right)^\prime\\
&-(2t_1\epsilon\rho^{-9}k_{1\ts\ts}+4t_2\epsilon\rho^{-7}k_{1\ts\tth}-2t_2\epsilon\rho^{-5}k_{1\tth\tth}+8t_1\epsilon^2\rho^{-7}\cos^2 (2\epsilon\ts) k_{1\ts}+4t_2\epsilon^2\rho^{-5}\cos (2\epsilon\ts) k_{1\tth})\frac{V'}{\tilde r},
\end{aligned}
\end{equation*}
\begin{equation*}
\begin{aligned}
Q_6:=&-[(2f_1\rho^{-9}+3\tilde r^2\rho^{-12}) k_{1\ts\ts}+4 f_2\rho^{-7}k_{1\ts\tth}+(-2 f_2\rho^{-5}+3\tilde r^2 \rho^{-8}) k_{1\tth\tth}]\epsilon^2\frac{V'}{\tilde r}.
\end{aligned}
\end{equation*}
Note that $Q_3 da$ and $Q_5 db$ are orthogonal to both $T_{B_1}$ and $T_{B_2}$ up to $O(\epsilon^2)$. It follows that
\begin{equation}\label{IntDeltaHB}
\begin{aligned}
&\int_{\R^2}\langle \Delta_H B,T_{B_1}\rangle
=-k_1\int_{0}^{+\infty}\int_{0}^{2\pi}\frac{V'}{\tilde r}\left(\frac{V'''}{\tilde r^2}-\frac{V''}{\tilde r^3}+\frac{V'}{r^4}\right)d\tilde \phi d\tilde r+4\pi \epsilon^2\rho^{-6} k_1\int_{0}^{+\infty}V'V'' d\tilde r\\
+&2(-\Delta_{\Gamma_{\epsilon}}k_1+2\epsilon\rho^{-2}\cos 2s k_{2\tth}+\epsilon^2\rho^{-2}\cos^2 2s k_1+O(\epsilon^2\nabla^2_{\Ge}k_1))\pi \int_{0}^{+\infty}\frac{(V')^2}{\tilde r}d\tilde r+O(\epsilon^3)\\
=&-2k_1\pi\int_{0}^{+\infty}\frac{V'}{\tilde r}\left(\frac{V'''}{\tilde r^2}-\frac{V''}{\tilde r^3}+\frac{V'}{\tilde r^4}\right)d\tilde r\\
+&2(-\Delta_{\Gamma_{\epsilon}}k_1+2\epsilon\rho^{-2}\cos 2s k_{2\tth}+\epsilon^2\rho^{-2}\cos^2 2s k_1+O(\epsilon^2\nabla^2_{\Ge}k_1))\pi \int_{0}^{+\infty}\frac{(V')^2}{\tilde r}d\tilde r+O(\epsilon^3\rho^{-4}).
\end{aligned}
\end{equation}

To calculate the second term, we decompose $\nabla_{A}\psi=\nabla_{A_0}\psi_0+\nabla_{A_0}\eta_0-iB_0\psi_0-B_0\eta_0$. And we have
\begin{equation*}
\begin{aligned}
\langle 2\Im(\overline{\nabla_{A_0}\psi_0}\cdot \xi), T_{B_1}\rangle
=&2k_1\Im \langle T_1\overline{\nabla_{A_0}\psi_0}, T_{B_1}\rangle+2k_2\Im \langle T_2\overline{\nabla_{A_0}\psi_0}, T_{B_1}\rangle=2k_1\frac{V'}{\tilde r}\Im(T_1 \overline{T_2})+O(\epsilon^3).
\end{aligned}
\end{equation*}
Integrating in $(t_1,t_2)$, we have
\begin{equation*}
\begin{aligned}
\int_{\R^2}\langle 2\Im(\overline{\nabla_{A_0}\psi_0}\cdot \xi), T_{B_1}\rangle
=&-4k_1\pi\int_{0}^{+\infty}\frac{1}{\tilde r}UU'V'(1-V)d\tilde r+O(\epsilon^3\rho^{-4})
\end{aligned}
\end{equation*}
and hence,
\begin{equation}\label{IntImDApsi}
\begin{aligned}
\int_{\R^2}\langle 2\Im(\overline{\nabla_A\psi}\cdot \xi), T_{B_1}\rangle
&=-4k_1\pi\int_{0}^{+\infty}\frac{1}{\tilde r}UU'V'(1-V)d\tilde r+\int_{\R^2}\Im\langle 2(\nabla_{A_0}\eta_0-iB_0\psi_0-B_0\eta_0)\xi, T_{B_1}\rangle.
\end{aligned}
\end{equation}

Finally, direct computations give
\begin{equation*}
\begin{aligned}
\langle B|\psi_0|^2, T_{B_1}\rangle
&=k_1U^2|T_{B_1}|^2+k_2U^2\langle T_{B_2}, T_{B_1}\rangle+O(\epsilon^3\rho^{-4}).
\end{aligned}
\end{equation*}
Integrating in $(t_1,t_2)$ yields
\begin{equation}\label{IntBpsi2}
\begin{aligned}
\int_{\R^2}\langle B|\psi|^2, T_{B_1}\rangle=2\pi k_1\int_{0}^{+\infty}U^2\frac{V'^2}{\tilde r}d\tilde r+\int_{\R^2}\langle B(|\psi|^2-|\psi_0|^2), T_{B_1}\rangle+O(\epsilon^3\rho^{-4}).
\end{aligned}
\end{equation}

Taking derivative on the second equation in \eqref{ODEsystem}, we get
\begin{equation*}
-V'''-\frac{V''}{\tilde r}+\frac{V'}{r^2}-2UU'(1-V)+U^2V'=0.
\end{equation*}
Combining it with \eqref{IntDeltaHB}, \eqref{IntImDApsi} and     \eqref{IntBpsi2}, we obtain the desired result.
\end{proof}

With Lemma \ref{InnerProduct1} and Lemma \ref{InnerProduct2}, we have the following rough result.
\begin{lemma}
\begin{equation}\label{RoughProj}
\begin{aligned}
    &\int_{(a,b)\in \Sigma_\epsilon}\langle \mathbb L(v;\psi,A), \mathcal T_1\rangle dt_1dt_2=\Re\int_{(a,b)\in \Sigma_\epsilon} \mathbb L(v;\psi,A)_1 \overline{T_1} dt_1dt_2+\int_{(a,b)\in \Sigma_\epsilon}\langle \mathbb L(v;\psi,A)_2, T_{B_1}\rangle dt_1dt_2\\
    =&-[\Delta_{\Gamma_{\epsilon}}k_1-2\epsilon\rho^{-2}\cos 2sk_{2\tth}+\epsilon^2\rho^{-2}(2\rho^{-4}-\cos^2 (2\epsilon\ts))k_1+O(\epsilon^2\nabla^2_{\Ge}k_1)]\int_{\R^2}|T_1|^2+|T_{B_1}|^2\\
    +&4\epsilon^2 \rho^{-6} k_1\pi\int_{0}^{+\infty}\frac
    {(V')^2}{\tilde r}d \tilde r+k_2\Im\int_{\R^2}d^*B_0 T_2\overline{T_1}
-2\Im\int_{\R^2}\langle B_0,d\xi\rangle \overline{T_1}
+\Re\int_{\R^2}(|A^{(1)}+B_0|^2-|A^{(1)}|^2)k_1|T_1|^2\\
+&\Re\int_{\R^2}2i\langle B, \nabla_{A_0}\eta_0-iB_0\psi_0-B_0\eta_0\rangle\overline{T_1}+\Re\int_{\R^2}[\frac12 (\lambda-1)(\psi^2-\psi_0^2)\bar\xi+(\lambda+\frac12)(|\psi|^2-|\psi_0|^2)\xi]\overline{T_1}\\
+&\int_{\R^2}\Im\langle 2(\nabla_{A_0}\eta_0-iB_0\psi_0-B_0\eta_0)\cdot \xi, T_{B_1}\rangle+\int_{\R^2}\langle B(|\psi|^2-|\psi_0|^2), T_{B_1}\rangle+O(\epsilon^3(|\partial_{ij}k|+|\partial_{i}k|+|k|)).
\end{aligned}
\end{equation}
\end{lemma}
It remains to compute the terms from the nonlinear terms. From Section 2, we see that the improvement of approximation has the form:
\begin{equation*}
\begin{aligned}
(\eta_0,B_0)
=
-2\epsilon^2\rho^{-6}
(\eta_1,B_1)
+O(\epsilon^3\rho^{-4}e^{-\delta r})
\end{aligned}
\end{equation*}
where $(\eta_1,B_1)$ solves
\begin{equation}\label{NonlinearFixRHS12}
    \mathbb L(\eta_1,B_1;\psi^{(1)},A^{(1)})=
    \begin{pmatrix}
        \tilde rU'e^{i\tilde\phi}\\
        -\sin\tilde\phi V' dt_1+\cos\tilde\phi V' dt_2
    \end{pmatrix}
    \text{ in }\R^2.
\end{equation}

Taking $\partial_{t_i}$ on $\mathbb L$, we have for any $(\mu,C)$,
\begin{equation}\label{translationfix}
\begin{aligned}
&\partial_{t_i}\mathbb L(\mu,C;\psi_0,A_0)
=\mathbb L(\partial_{t_i}\mu,\partial_{t_i}C;\psi_0,A_0)
\\
&+
\begin{pmatrix}
2i\langle \partial_{t_i}A_0, \nabla_{A_0}\mu \rangle+2i\langle C, \nabla_{A_0}\partial_{t_i}\psi_0-i\partial_{t_i}A_0\psi_0\rangle+(\lambda-1)\psi_0\partial_{t_i}\psi_0 \overline{\mu}+(2\lambda+1)\Re(\overline{\psi_0}\partial_{t_i}\psi_0)\mu\\
2\Im[\overline{\nabla_{A_0}\partial_{t_i}\psi_0-i\partial_{t_i}A_0\psi_0}\cdot\mu]+2\Re(\overline{\psi_0}\partial_{t_i}\psi_0)C
\end{pmatrix}\\
&+\begin{pmatrix}
-\partial_{t_i}(\frac{1}{\sqrt{G}}\partial_k(\sqrt{G}g^{jk}))(\nabla_{A_0}\mu)_j-\partial_{t_i}g^{jk}\partial_{jk}\mu+2i\partial_{t_i}g^{jk}A_{0k}\partial_j\mu+\partial_{t_i}g^{jk}A_{0j}A_{0k}\mu+2i\partial_{t_i}g^{jk}C_j(\nabla_{A_0}\psi_0)_k\\
[-\partial_{t_i,m}(\frac{1}{\sqrt{G}}\partial_k(\sqrt{G}g^{jk}))C_j-\partial_{t_i,k}g^{jk}\partial_jC_m-\partial_{t_i}g^{jk}\partial_{jk}C_m-\partial_{t_i}(g_{mn}g^{jk}\partial_k g^{ln})C_{jl}]dy^m
\end{pmatrix}.
\end{aligned}
\end{equation}

Let us recall that $\mathbb L(\cdot ;\psi_0,A_0)$ is invariant under the following gauge transformation
\begin{equation*}
\mathbb L(\tilde G_{\gamma}(\xi,B);G_{\gamma}(\psi_0,A_0))=\tilde G_{\gamma}(\mathbb L(\xi,B))\text{ for any }\gamma \in C^1(\R^4).
\end{equation*}
Then we have
\begin{equation}\label{gaugefix}
\begin{aligned}
&\frac{d}{dt}|_{t=0}\tilde G_{t\gamma}(\mathbb L(\mu,C))=\frac{d}{dt}|_{t=0}\mathbb L(\tilde G_{t\gamma}(\mu,C);G_{t\gamma}(\psi_0,A_0))=\mathbb L(i\gamma\mu,C;\psi_0,A_0)\\
&+
\begin{pmatrix}
-id^*d\gamma \mu +2i\langle d\gamma, \nabla_{A_0}\mu\rangle+2i\langle C, \nabla_{A_0}(i\gamma\psi_0)-id\gamma \psi_0\rangle+(\lambda-1)i\gamma \psi_0^2\overline{\mu}\\
2\Im[\overline{\nabla_{A_0}(i\gamma \psi_0)-id\gamma \psi_0)}\cdot \mu]
\end{pmatrix}.
\end{aligned}
\end{equation}

Combining \eqref{translationfix}, \eqref{gaugefix} and taking $\gamma=A_3$ , for example, we have
\begin{equation*}
\begin{aligned}
&\partial_{t_i}\mathbb L(\mu,C;\psi_0,A_0)-\frac{d}{dt}|_{t=0}\tilde G_{tA_3}(\mathbb L(\mu,C))
=\mathbb L(\partial_{t_i}\mu,\partial_{t_i}C)-\mathbb L(iA_3\mu,C;\psi_0,A_0)\\
+&
\begin{pmatrix}
-id^*T_{B_1}\mu +2i\langle T_{B_1}, \nabla_{A_0}\mu\rangle+2i\langle C,\nabla_{A_0}T_{1}-iT_{B_1}\psi_0 \rangle+(\lambda-1)\psi_0 T_1 \overline{\mu}+(2\lambda+1)\Re(\overline{\psi_0}T_1)\mu\\
2\Im[\overline{\nabla_{A_0}T_1-iT_{B_1}\psi_0}\cdot \mu]+2\Re(\overline{\psi_0}T_1)C
\end{pmatrix}\\
+&\begin{pmatrix}
-\partial_{t_i}(\frac{1}{\sqrt{G}}\partial_k(\sqrt{G}g^{jk}))(\nabla_{A_0}\mu)_j-\partial_{t_i}g^{jk}\partial_{jk}\mu+2i\partial_{t_i}g^{jk}A_{0k}\partial_j\mu+\partial_{t_i}g^{jk}A_{0j}A_{0k}\mu+2i\partial_{t_i}g^{jk}C_j(\nabla_{A_0}\psi_0)_k\\
[-\partial_{t_i,m}(\frac{1}{\sqrt{G}}\partial_k(\sqrt{G}g^{jk}))C_j-\partial_{t_i,k}g^{jk}\partial_jC_m-\partial_{t_i}g^{jk}\partial_{jk}C_m-\partial_{t_i}(g_{mn}g^{jk}\partial_k g^{ln})C_{jl}]dy^m
\end{pmatrix}.
\end{aligned}
\end{equation*}

Plugging in $(\mu,C)=(\eta_1,B_1)$ and testing the above equation by $(T_1,T_{B_1})$, and integrating in $(t_1,t_2)$, we have
\begin{equation*}
\begin{aligned}
&-2\Im\int_{\R^2}\langle B_1, dT_1 \rangle+2\Re\int_{\R^2}\langle B_1, A_0\rangle|T_1|^2+\Re\int_{\R^2}2i\langle T_{B_1}, \nabla_{A_0}\eta_1-iB_1 \psi_0\rangle\overline{T_1}+O(\epsilon)\\
+&\Re \int_{\R^2} (\lambda-1)\psi_0\eta_1\overline{T_1}+(2\lambda+1)\psi_0\eta_1|T_1|^2+\int_{\R^2}\Im\langle 2(\nabla_{A_0}\eta_1-iB_i\psi_0)T_1, T_{B1} \rangle+2\int_{\R^2} \langle 2\psi_0 \eta_1 T_{B_1}, T_{B_1}\rangle\\
=&\int_{\R^2}\langle \partial_{t_1}L_0(\mu,C;\psi_0,A_0)-\frac{d}{dt}|_{t=0}\tilde G_{tA_3}(L_0(\mu,C)),(T_1, T_{B_1})\rangle=:E.
\end{aligned}
\end{equation*}

Plugging \eqref{NonlinearFixRHS12} into the right-hand side above, we can compute $E$ as
\begin{equation*}
E=\pi \int_{0}^{+\infty}\tilde r(\tilde rU''+U')U'+UU'(1-V)^2+\frac{(V')^2}{\tilde r}d\tilde r+O(\epsilon)=2\pi \int_{0}^{+\infty}\frac{(V')^2}{\tilde r}d\tilde r+O(\epsilon).
\end{equation*}
Then \eqref{RoughProj} can be represented as
\begin{equation*}
\begin{aligned}
    &\int_{(a,b)\in \Sigma_\epsilon}\langle \mathbb L(v;\psi,A), \mathcal T_1\rangle dt_1dt_2=\Re\int_{(a,b)\in \Sigma_\epsilon} \mathbb L(v;\psi,A)_1 \overline{T_1} dt_1dt_2+\int_{(a,b)\in \Sigma_\epsilon}\langle \mathbb L(v;\psi,A)_2, T_{B_1}\rangle dt_1dt_2\\
    =&-[\Delta_{\Gamma_{\epsilon}}k_1-2\epsilon^2\rho^{-2}\cos 2\epsilon \ts k_{2\theta}+\epsilon^2\rho^{-2}(2\rho^{-4}-\cos^2 2\epsilon \ts)k_1+O(\epsilon^2\nabla^2_{\Ge}k_1)]\int_{\R^2}|\mathcal T_1|^2\\
    +&4\epsilon^2 \rho^{-6} k_1\pi\int_{0}^{+\infty}\frac
    {(V')^2}{\tilde r}d \tilde r+2\epsilon^2 \rho^{-6} k_1 E-2\Im\int_{\R^2}\langle B_0,(k_{1\theta}T_1+k_{2\theta}T_2)d\theta\rangle \overline{T_1}+O(\epsilon^3\rho^{-4})\\
    =&-[\Delta_{\Gamma_{\epsilon}}k_1-2\epsilon^2\rho^{-2}\cos 2\epsilon \ts k_{2\theta}+\epsilon^2\rho^{-2}(2\rho^{-4}-\cos^2 2\epsilon \ts)k_1+O(\epsilon^2\nabla^2_{\Ge}k_1)]\int_{\R^2}|\mathcal T_1|^2+O(\epsilon^3\rho^{-4}),
\end{aligned}
\end{equation*}
which is exactly the first component appeared in Proposition \ref{TestKernel}. Similar computation gives the second component and we finish the proof of Proposition \ref{TestKernel}.

Before we state the next proposition, we define a cut-off function $\chi$ supported in $\Sigma_{\epsilon}$:
\begin{equation*}
\begin{cases}
    \chi=1&\text{ if }a^2+b^2\leq \frac12 r_\epsilon,\\
    \chi=0&\text{ if }a^2+b^2\geq r_\epsilon,
\end{cases}
\end{equation*}
where we recall that $r_\epsilon(\ts,\tth)=\frac12 \epsilon^{-\frac 12}\rho(\ts,\tth)^2$. We also define the region
\begin{equation*}
\mathcal W^R:=\{x\in \Sigma_{\epsilon}:\ \rho(\ts,\tth)<R\}
\end{equation*}
and
\begin{equation*}
\Gamma_{\epsilon}^R:=\{(\ts,\tth)\in \Gamma_{\epsilon}:\ \rho(\ts,\tth)<R\}
\end{equation*}
for some large constant $R>0$.

\begin{proposition}\label{EnergyEstiamte}
Let $N=k_1(\ts,\tth)\m+k_2(\ts,\tth)\n$ be a normal vector field that vanishes when $\rho(\ts,\tth)=R$ and set 
\begin{equation*}
v(x)=v(\ts,\tth,t_1,t_2)=(\xi,B)(\ts,\tth,t_1,t_2)= k_1(\ts,\tth)\mathcal T_1(t_1,t_2)+ k_2(\ts,\tth)\mathcal T_2(t_1,t_2).
\end{equation*}
Then we have the following estimate
\begin{equation}\label{SecondVariationRelation}
\begin{aligned}
&\mathcal{Q}(\chi v,\chi v):
=\int_{\mathcal W^R}|\nabla_A\xi|^2+|dB|^2+|d^*B|^2+|B|^2|\psi|^2+4\langle \Im(\overline{\nabla_A\psi}\xi), B\rangle\\
&+\frac{\lambda-1}{2}\Re(\overline{\psi}\xi)^2+(\lambda+\frac12)|\psi|^2|\xi|^2-\frac{\lambda}{2}|\xi|^2 dx\\
=& \int_{\Gamma_{\epsilon}^R} |\nabla^{\nu}_{\Gamma_\epsilon}N|^2-2\epsilon^2\rho^{-6} |N|^2 dvol_{\Gamma_\epsilon}\cdot \int_{\R^2}|\mathcal T_1|^2+O\left(\epsilon \int_{\Gamma_{\epsilon}^R} |\nabla N|^2+\epsilon^2\rho^{-6} |N|^2 dvol_{\Gamma_\epsilon}\right),
\end{aligned}
\end{equation}
where $\nabla^{\nu}_{\Gamma_\epsilon}N$ is the covariant derivative on the normal bundle $\mathcal{N}\Gamma_\epsilon$.
\end{proposition}

\begin{proof}
Direct computation gives that
\begin{equation*}
dx=\sqrt{G}=\sqrt{\det\tilde A}(1-\epsilon^2 \tr \rho^{-6}).
\end{equation*}
Hence, we have
\begin{equation}\label{Testself}
\int_{\mathcal W^R}\langle \mathbb L(v;\psi,A), v\rangle dx=\int_{\Gamma_{\epsilon}^R}\int_{\tr<r_\epsilon}\langle \mathbb L(v;\psi,A), v\rangle (1-\epsilon^2 \tr^2 \rho^{-6}) da db d vol_{\Gamma_{\epsilon}}. 
\end{equation}

By Proposition \ref{TestKernel}, we find that
\begin{equation}\label{Testself1}
\begin{aligned}
&\int_{\Gamma_{\epsilon}^R}\int_{r<r_\epsilon}\langle \mathbb L(v;\psi,A), v\rangle
\\
=&-\int_{\Gamma_{\epsilon}^R}\langle L_{\Gamma_{\epsilon}}N, N\rangle dvol_{\Gamma_\epsilon}\cdot (1+O(\epsilon)) \int_{\R^2}|\mathcal T_1|^2 +\int_{\Gamma_{\epsilon}^R}O(\epsilon (D^2 k+Dk+k))\cdot k\\
=&-\int_{\Gamma_{\epsilon}^R}\langle \Delta^{\nu}_{\Gamma_{\epsilon}}N+2\epsilon^2 \rho^{-6}N, N\rangle dvol_{\Gamma_\epsilon}\cdot (1+O(\epsilon)) \int_{\R^2}|\mathcal T_1|^2+\int_{\Gamma_{\epsilon}^R}O(\epsilon (D^2 k+Dk+k))\cdot k.
\end{aligned}  
\end{equation}
Integrating by parts yields
\begin{equation*}
\begin{aligned}
&\int_{\Gamma_{\epsilon}^R}\int_{r<r_\epsilon}\langle \mathbb L(v;\psi,A), v\rangle\\
=&\int_{\Gamma_{\epsilon}^R}|\nabla^{\nu}_{\Gamma_{\epsilon}}N|^2-2\epsilon^2 \rho^{-6}|N|^2 dvol_{\Gamma_\epsilon}\cdot \int_{\R^2}|\mathcal T_1|^2 +O\left(\epsilon \int_{\Gamma_{\epsilon}^R} |\nabla N|^2+\epsilon^2\rho^{-6} |N|^2 dvol_{\Gamma_\epsilon}\right).
\end{aligned}  
\end{equation*}

Since $\mathcal T_1$ and $\mathcal T_2$ decay exponentially on $\R^2$, similar computation gives that
\begin{equation}\label{Testself2}
\int_{\Gamma_{\epsilon}^R}\int_{r<r_\epsilon}\langle \mathbb L(v;\psi,A), v\rangle \epsilon^2 \tr^2 \rho^{-6} da db d vol_{\Gamma_{\epsilon}}=O\left(\epsilon \int_{\Gamma_{\epsilon}^R} |\nabla N|^2+\epsilon^2\rho^{-6} |N|^2 dvol_{\Gamma_\epsilon}\right).
\end{equation}

Combining \eqref{Testself1} and \eqref{Testself2}, we see that \eqref{Testself} can be written as
\begin{equation*}
\begin{aligned}
\mathcal Q(v,v)=\int_{\mathcal W^R}\langle \mathbb L(v;\psi,A), v\rangle dx&= \int_{\Gamma_{\epsilon}^R} |\nabla^{\nu}_{\Gamma_\epsilon}N|^2-2\epsilon^2\rho^{-6} |N|^2 dvol_{\Gamma_\epsilon} \int_{\R^2}|\mathcal T_1|^2\\
&+O\left(\epsilon \int_{\Gamma_{\epsilon}^R} |\nabla N|^2+\epsilon^2\rho^{-6} |N|^2 dvol_{\Gamma_\epsilon}\right).      
\end{aligned}  
\end{equation*}
Then the conclusion follows for $\mathcal Q(v,v)$. A similar computation holds for $\mathcal Q(\chi v,\chi v)$ since $\chi v$ vanishes on $\partial \mathcal W^R$ and we obtain the desired result.
\end{proof}
The proposition above builds up the relationship between $\mathcal{Q}$ and the second variation of the area functional on $\Gamma_{\epsilon}$. With the stability of $\Gamma_{\epsilon}$ and Proposition \ref{EnergyEstiamte}, it suffices to show that the eigenfunctions of $\mathbb L(\cdot; \psi, A)$ with respect to negative eigenvalues (if exist) are almost $k_1(\ts,\tth)\mathcal T_1+k_2(\ts,\tth)\mathcal T_2$. To show this, we need to derive several estimates. We first show that the negative eigenvalues of $\mathbb L(\cdot; \psi, A)$ (if exist) are $O(\epsilon^2)$.

\begin{lemma}\label{EigenBound}
There exists a constant $\mu_0>0$ independent of $R$ and sufficiently small $\epsilon$, such that if $\mu\leq 0$ is an eigenvalue of problem \eqref{GLEigenBR}, then
\begin{equation*}
    \mu\geq -\mu_0\epsilon^2.
\end{equation*}
\end{lemma}
\begin{proof}
we dentoe $\mathcal{Q}_{\Omega}$ be the restriction of $\mathcal{Q}$ in the domain $\Omega$:
\begin{equation*}
\begin{aligned}
&\mathcal{Q}_{\Omega}(v,v):
=\int_{\Omega}|\nabla_A\xi|^2+|dB|^2+|d^*B|^2+|B|^2|\psi|^2+4\langle \Im(\overline{\nabla_A\psi}\xi), B\rangle+\frac{\lambda-1}{2}\Re(\overline{\psi}\xi)^2+(\lambda+\frac12)|\psi|^2|\xi|^2-\frac{\lambda}{2}|\xi|^2
\end{aligned}
\end{equation*}
for $v=(\xi, B)$.

Recall that $(\psi, A)=(W e^{i\tilde \varphi},Z d\tilde \varphi)+O(\epsilon^2 e^{-\delta \tilde r})$ in $\R^4\setminus \Sigma_{\epsilon}$, where $W,Z=1-O(e^{-\delta \tilde r})$ for some $0<\delta<1$ and $\varphi$ is an extension of the angle function $\phi$ in $\R^4$. Then $(\psi, A)$ is stable in $\R^4\setminus \Sigma_{\epsilon}$. Now we set $\Omega:=\Sigma_{\epsilon}\cap \{\rho<R\}$. Then for any $v$, we have
\begin{equation*}
\mathcal Q(v,v)\geq \mathcal Q_{\Omega}(v,v)+\gamma \int_{\R^4\setminus \Sigma_{\epsilon}}|v|^2
\end{equation*}
for some $\gamma>0$ independent of $\epsilon$ and $R$. In the following, we will show that 
\begin{equation*}
\mathcal{Q}_{\Omega}(v,v)\geq -\mu_0 \epsilon^2 \int_{\Omega}\rho^{-6}|v|^2.
\end{equation*}
Its corresponding eigenvalue problem is
\begin{equation}\label{GLEigenSigma}
\begin{cases}
\mathbb L(v;\psi,A)-\mu \rho^{-6}v=0&\text{ in }\Omega\\
v=0&\text{ on }\{\rho=R\}\\
\partial_{\nu}\xi-i\langle A,\nu \rangle \xi=0,\ d^*B=0,\ *dB=0 &\text{ on }\partial \Sigma_{\epsilon}.
\end{cases}
\end{equation}

For an eigenfunction $v$ solving \eqref{GLEigenSigma}, we decompose $v$ as
\begin{equation*}
v=\chi k_1(\ts,\tth)\mathcal T_1+\chi k_2(\ts,\tth)\mathcal T_2+v^\perp,
\end{equation*}
where for $j=1,2$,
\begin{equation*}
k_j(\ts,\tth)=\frac{\int_{r<r_\epsilon}\langle v(\ts,\tth,t_1,t_2), \mathcal T_j(t_1,t_2)\rangle dt_1 dt_2}{\int_{\R^2}\chi|\mathcal T_1(t_1,t_2)|^2 dt_1 dt_2}.    
\end{equation*}

From our decomposition, we see that $k_j$ vanishes on $\{\rho=R\}$, $v^\perp$ satisfies the same boundary condition as $v$ and
\begin{equation}\label{vperporthgonal}
\int_{r<r_\epsilon}\langle v^\perp(\ts,\tth,t_1,t_2), \mathcal T_j(t_1,t_2)\rangle dt_1 dt_2=0,\text{ for }j=1,2\text{ and }(\ts,\tth)\in \Gamma_{\epsilon}^R.
\end{equation}

With the decomposition, we have
\begin{equation*}
\mathcal Q_{\Omega}(v,v)=\mathcal Q_{\Omega}(v^\perp,v^\perp)+\mathcal Q_{\Omega}(\chi k_1\mathcal T_1+k_2\mathcal T_2,\chi k_1\mathcal T_1+k_2\mathcal T_2)+2\mathcal Q_{\Omega}(v^\perp,\chi k_1\mathcal T_1+\chi k_2\mathcal T_2))
\end{equation*}
and each term will be estimated in the following. For simplicity, we denote $d_{\R^2}$, $d^{*}_{\R^2}$, $\nabla_{A^{(1)},\R^2}:=d_{\R^2}-iA^{(1)}$ to be the exterior derivative, codifferential and connection gradient on the normal space $(t_1,t_2)$. And $d_{\Gamma_{\epsilon}}$, $d^{*}_{\Gamma_{\epsilon}}$, $\nabla_{\Gamma_{\epsilon}}$ be the ones on $\Gamma_{\epsilon}$.

For $Q_{\Omega}(v^\perp,v^\perp)$, we write $v^\perp=(\xi^{\perp},B^{\perp})$ and compute
\begin{equation}\label{Qxiperplowerbound}
\begin{aligned}
\int_{\Omega}|\nabla_A \xi^\perp|^2
&=\int_{\Omega}|\nabla_{A^{(1)},\R^2}\xi^\perp|^2+(1+O(\epsilon^2))|\nabla_{\Ge}\xi^\perp|^2+O(\epsilon)|\nabla_{\Ge}\xi^\perp||\nabla_{A^{(1)},\R^2}\xi^{\perp}|\\
&+O(\epsilon^2)|\nabla_{A^{(1)},\R^2}\xi^{\perp}|^2+O(\epsilon^2)|\xi^{\perp}|^2 dx\\
&\geq \int_{\rho<R}\int_{r<r_{\epsilon}}(1+O(\epsilon^2))(|\nabla_{A^{(1)},\R^2}\xi^\perp|^2+|\nabla_{\Ge}\xi^\perp|^2)+O(\epsilon^2 )|\xi^\perp|^2 dt_1 dt_2 d Vol_{\Ge}.
\end{aligned}
\end{equation}
We decompose $B^{\perp}=(B^{\perp}_1 d\ts+B^{\perp}_2 d\tth) +(B^{\perp}_3 dt_1+B^{\perp}_4 dt_2)=:B^\perp_{\Ge}+B^\perp_{\R^2}$, then the second and third terms in $Q_{\Omega}(v^\perp, v^{\perp})$ can also be computed as
\begin{equation*}
\begin{aligned}
&\int_{\Omega}|dB^\perp|^2+|d^*B^\perp|^2 dx\\
=&\int_{\Omega}\frac12 \sum_{j,k,l,m=1}^{4}\tg^{jk}\tg^{lm}B_{jl}^\perp B_{km}^\perp+\sum_{j,k=1}^{4}\left|\frac{1}{\sqrt{\tG}}\partial_k(\tg^{jk}\sqrt{\tG}B^\perp_j)\right|^2 dx\\
=&\int_{\Omega}(1+O(\epsilon^2))(|d_{\R^2}B^\perp_{\R^2}|^2+|d_{\R^2}^* B^\perp_{\R^2}|^2+|d_{\Ge}B^\perp_{\Ge}|^2+|d_{\Ge}^* B^\perp_{\Ge}|^2)\\
&+\sum_{j=1}^{2}\sum_{k=3}^{4}\tg^{jj}\tg^{kk}(B^\perp_{jk})^2+2(\rho^{-6}B^\perp_{1\ts}+\rho^{-2}B^\perp_{2\tth}+2\epsilon \rho^{-4}\cos (2\epsilon \ts)  B^\perp_1)(B^{\perp}_{3t_1}+B^\perp_{4t_2}) +O(\epsilon)|\nabla_{\Ge} B^\perp||\nabla_{\R^2}B^\perp| dx\\
=&\int_{\rho<R}\int_{r<r_{\epsilon}}(1+O(\epsilon^2))(|d_{\R^2}B^\perp_{\R^2}|^2+|d_{\R^2}^* B^\perp_{\R^2}|^2+|d_{\Ge}B^\perp_{\Ge}|^2+|d_{\Ge}^* B^\perp_{\Ge}|^2)\rho^{4}\\
&+\rho^{-2}((B^\perp_{3\ts})^2+(B^\perp_{4\ts})^2+(B^\perp_{1t_1})^2+(B^\perp_{1t_2})^2)+\rho^{2}((B^\perp_{3\tth})^2+(B^\perp_{4\tth})^2+(B^\perp_{2t_1})^2+(B^\perp_{2t_2})^2)\\
&-2\rho^{-2}(B^\perp_{1t_1}B^\perp_{3\ts}+B^\perp_{1t_2}B^\perp_{4\ts})-2\rho^{2}(B^\perp_{2t_1}B^\perp_{3\tth}+B^\perp_{2t_2}B^\perp_{4\tth})\\
&+2(\rho^{-2}B^\perp_{1\ts}+\rho^{2}B^\perp_{2\tth}+2\epsilon \cos (2\epsilon \ts) B^\perp_1)(B^{\perp}_{3t_1}+B^\perp_{4t_2}) +O(\epsilon)|\nabla_{\Ge} B^\perp||\nabla_{\R^2}B^\perp| dt_1 dt_2 d\ts d\theta .
\end{aligned}
\end{equation*}
Integrating by parts, we found that the cross terms are canceled. Hence,
\begin{equation}\label{QBperplowerbound}
\begin{aligned}
\int_{\Omega}|dB^\perp|^2+|d^*B^\perp|^2 dx\geq (1+O(\epsilon))\int_{\rho<R}\int_{r<r_{\epsilon}}|d_{\R^2}B^\perp_{\R^2}|^2+|d_{\R^2}^* B^\perp_{\R^2}|^2+|d_{\Ge}B^\perp_{\Ge}|^2+|d_{\Ge}^* B^\perp_{\Ge}|^2 dt_1 dt_2 dVol_{\Ge}. 
\end{aligned}    
\end{equation}

Combining \eqref{Qxiperplowerbound} and \eqref{QBperplowerbound}, we have
\begin{equation}\label{Qvperprough}
\begin{aligned}
\mathcal{Q}_{\Omega}(v^\perp,v^\perp)\geq& (1+O(\epsilon))\int_{\rho<R}\int_{r<r_{\epsilon}}\left[|\nabla_{A^{(1)},\R^2}\xi^\perp|^2+|d_{\R^2}B^\perp_{\R^2}|^2+|d_{\R^2}^* B^\perp_{\R^2}|^2+|B^\perp_{\R^2}|^2|\psi^{(1)}|^2 \right.\\
&\left.+4 \Im(\overline{\nabla_{A^{(1)}}\psi^{(1)}}\xi^\perp)\cdot B^\perp_{\R^2}+\frac{\lambda-1}{2}\Re(\overline{\psi^{(1)}}\xi^\perp)^2+(\lambda+\frac12)|\psi^{(1)}|^2|\xi^\perp|^2-\frac{\lambda}{2}|\xi|^2\right.\\
&\left.+|\nabla_{\Ge}\xi^\perp|^2 +|d_{\Ge}B^\perp_{\Ge}|^2+|d_{\Ge}^* B^\perp_{\Ge}|^2+O(\epsilon^2)|\xi^\perp|^2 \right]dt_1 dt_2 dVol_{\Gamma_{\epsilon}}.    
\end{aligned}  
\end{equation}

Recall that the $+1$-vortex solution $(\psi^{(1)},A^{(1)})$ is stable, then one can show that for any $a>0$, if $(\mu, C)$ is perpendicular to $(T_j, T_{B_j})$ in $B_a$ for $j=1,2$, then a standard contradiction argument (see \cite{BM1,BM2,BM3} for example) gives that 
\begin{equation*}
\begin{aligned}
&\int_{r<a}|\nabla_{A^{(1)},\R^2}\mu|^2+|d_{\R^2}C|^2+|d_{\R^2}^* C|^2+|C|^2|\psi^{(1)}|^2+4\Im(\overline{\nabla_{A^{(1)}}\psi^{(1)}}\mu)\cdot C+\frac{\lambda-1}{2}\Re(\overline{\psi^{(1)}}\mu)^2\\
&+(\lambda+\frac12)|\psi^{(1)}|^2|\mu|^2-\frac{\lambda}{2}|\mu|^2 dt_1 dt_2 \\
\geq& 3\gamma \int_{r<a}|\nabla_{A^{(1)},\R^2}\mu|^2+|d_{\R^2}C|^2+|d_{\R^2}^* C|^2+|\mu|^2+|C|^2
\end{aligned}
\end{equation*}
for some $\gamma>0$. Plugging it into \eqref{Qvperprough}, we have
\begin{equation}\label{QOmega1}
\mathcal{Q}_{\Omega}(v^\perp,v^\perp)\geq 2\gamma\int_{\Omega} |\nabla_{A^{(1)},\R^2}\xi^\perp|^2+|d_{\R^2}B^\perp_{\R^2}|^2+|d_{\R^2}^* B^\perp_{\R^2}|^2+|\xi^\perp|^2+|B^\perp_{\R^2}|^2+|\nabla_{\Ge}\xi^\perp|^2 +|d_{\Ge}B^\perp_{\Ge}|^2+|d_{\Ge}^* B^\perp_{\Ge}|^2 dx.
\end{equation}

Next, we will estimate the cross-term $\mathcal Q_{\Omega}(v^\perp,\chi k_1\mathcal T_1+\chi k_2\mathcal T_2))$, which can be written as
\begin{equation*}
\mathcal Q_{\Omega}(v^\perp,\chi k_1\mathcal T_1+\chi k_2\mathcal T_2))=\int_{\Omega}\langle -\mathbb L(\chi k_1\mathcal T_1+\chi k_2\mathcal T_2;\psi, A), v^\perp \rangle=I_1+I_2,
\end{equation*}
where
\begin{equation*}
I_1:=\int_{\Omega}\chi \langle -\mathbb L(k_1\mathcal T_1+k_2\mathcal T_2;\psi, A), v^\perp \rangle dx,
\end{equation*}
and
\begin{equation*}
\begin{aligned}
I_2:=\int_{\Omega}\langle 
\begin{pmatrix}
        2\nabla \chi\cdot \nabla(k_1T_1+k_2T_2)+(k_1T_1+k_2T_2)\Delta \chi-2i\langle A,d\chi \rangle (k_1T_1+k_2T_2) \\
        2\nabla \chi\cdot \nabla(k_1T_{B_1}+k_2T_{B_2})+(k_1T_{B_1}+k_2T_{B_2})\Delta \chi
    \end{pmatrix}
, v^\perp \rangle dx.
\end{aligned}
\end{equation*}
Since $k_j$ decays exponentially in the normal direction and $O(\rho^{-\delta})$ on $\Ge$ for some $0<\delta<\frac14$,
\begin{equation*}
I_2=o(1)\int_{\Ge^R}|\nabla_{\Ge}^\nu N|^2+\epsilon^2 \rho^{-6}|N|^2 dVol_{\Ge}+o(1)\int_{\Omega}|v^\perp|^2+|\nabla v^\perp|^2 dt_1 dt_2=:\mathbf o.
\end{equation*}
For $I_1$, similar computations as in the proof of Proposition \ref{TestKernel} show that
\begin{equation*}
I_1=\int_{\Omega}\chi \langle \Delta_{\Ge} k_1 \mathcal T_1+\Delta_{\Ge} k_2 \mathcal T_2-(\mathbb L(k_1\mathcal T_1+k_2\mathcal T_2;\psi,A)-\mathbb L(k_1\mathcal T_1+k_2\mathcal T_2;\psi_0,A_0)),v^\perp \rangle+\mathbf o.
\end{equation*}

The orthogonality condition \eqref{vperporthgonal} of $v^\perp$ implies that
\begin{equation*}
\begin{aligned}
\int_{\Omega}\chi \langle \Delta_{\Ge} k_1 \mathcal T_1+\Delta_{\Ge} k_2 \mathcal T_2, v^\perp\rangle
=&-\int_{\Omega}(1-\chi) \langle \Delta_{\Ge} k_1 \mathcal T_1+\Delta_{\Ge} k_2 \mathcal T_2,v^\perp \rangle dx\\
=&\int_{\Omega}\nabla_{\Ge}k_1\cdot [-\nabla_{\Ge}\chi\langle  \mathcal T_1,v^\perp \rangle+(1-\chi)\langle \nabla_{\Ge}\mathcal T_1,v^\perp \rangle+(1-\chi)\langle \mathcal T_1,\nabla_{\Ge} v^\perp \rangle]\\
&+\nabla_{\Ge}k_2\cdot [-\nabla_{\Ge}\chi\langle  \mathcal T_2,v^\perp \rangle+(1-\chi)\langle \nabla_{\Ge}\mathcal T_2,v^\perp \rangle+(1-\chi)\langle \mathcal T_2,\nabla_{\Ge} v^\perp \rangle] dx\\
=&\mathbf o 
\end{aligned}
\end{equation*}
due to the exponential decay of $\mathcal T_1$ and $\mathcal T_2$. For the second term (nonlinear terms), since $\eta_0, B_0=O(\epsilon^2\rho^{-4}e^{-\delta r})$, we have that
\begin{equation*}
\begin{aligned}
&\int_{\Omega}\chi\langle \mathbb L(k_1\mathcal T_1+k_2\mathcal T_2;\psi,A)-\mathbb L(k_1\mathcal T_1+k_2\mathcal T_2;\psi_0,A_0),v^\perp \rangle\geq -C\epsilon^2\int_{\Omega}\rho^{-6}|N||v^\perp|dx\\
\geq& -C\epsilon^2\nu \int_{\Ge}\rho^{-6}|N|^2 dvol_{\Ge}-C\epsilon^2 \nu^{-1}\int_{\Omega}\rho^{-6}|v^\perp|^2 dx.
\end{aligned}
\end{equation*}
Hence, we have
\begin{equation}\label{QOmega2}
\mathcal Q_{\Omega}(v^\perp,\chi k_1\mathcal T_1+\chi k_2\mathcal T_2))\geq -C\epsilon^2\nu^{-1} \int_{\Ge}\rho^{-6}|N|^2 dvol_{\Ge}-C\epsilon^2 \nu\int_{\Omega}\rho^{-6}|v^\perp|^2 dx.
\end{equation}
Finally, by Proposition \eqref{EnergyEstiamte}, we see that
\begin{equation}\label{QOmega3}
Q_{\Omega}(\chi k_1\mathcal T_1+k_2\mathcal T_2,\chi k_1\mathcal T_1+k_2\mathcal T_2)=\int_{\rho<R}|\nabla_{\Ge}^{\nu}N|^2-2\epsilon^2\rho^{-6}|N|^2 dvol_{\Ge} +\mathbf o.    
\end{equation}

Combining \eqref{QOmega1}, \eqref{QOmega2} and \eqref{QOmega3}, we find that if $\nu$ is sufficiently small, then
\begin{equation*}
Q_{\Omega}(v,v)\geq -C\epsilon^2 \int_{\Ge}\rho^{-6}|N|^2 dvol_{\Ge}\geq -\mu_0 \epsilon^2 \int_{\Omega}\rho^{-6}|v|^2,
\end{equation*}
which is our desired estimate.
\end{proof}

\section{Proof of Theorem \ref{main}}
In this section, we finish the proof of the main theorem, based on the estimates established in the previous sections. The strategy is essentially a contradiction argument. We assume to the contrary that there existed a sequence $\epsilon_n\to 0$ and $R_n\to +\infty$ such that \eqref{GLEigenBR} admits a sequence of eigenfunctions $v_{\epsilon_n, R_n}$ associated to the corresponding negative eigenvalues $\mu_{\epsilon_n, R_n}<0$. We will show that $v_{\epsilon_n, R_n}$ can be written as $v_{\epsilon_n, R_n}\thickapprox k_{1,\epsilon_n, R_n}(\ts,\tth)\mathcal T_1+k_{2,\epsilon_n, R_n}(\ts,\tth)\mathcal T_2$. Moreover, $N_{\epsilon_n, R_n}:=k_{1,\epsilon_n, R_n}\m+k_{2,\epsilon_n, R_n}\n$ is almost an eigenfunction of the Jacobi operator $L_{\Gamma_{\epsilon_n}}$, associated to the negative eigenvalue $\mu_{\epsilon_n, R_n}$. This contradicts with the stability of $\Gamma_{\epsilon_n}$. 

For simplicity, we will omit the subscript $n$. We may further assume that $\|v_{\epsilon,R}\|_{L^{\infty}}=1$. By the variational characterization of the eigenvalue, we can further assume $\mu_{\epsilon,R}$ is monotone nonincreasing, tending to some $\mu_\epsilon<0$ as $R\to +\infty$. In view of Lemma \ref{EigenBound}, the eigenvalues $\mu_{\epsilon,R}=O(\epsilon^2)$ and we denote $\mu_{\epsilon,R}=\hat \mu_{\epsilon,R}\epsilon^2$ and $\mu_{\epsilon}=\hat \mu_{\epsilon}\epsilon^2 <0$. In the following, we show that up to a subsequence, $v_{\epsilon,R}\to v_{\epsilon}$ for some eigenfunction $v_\epsilon$ of problem \eqref{GLEigenR4} associated to eigenvalue $\hat \mu_\epsilon \epsilon^2$.

Since Lemma \ref{EigenBound} shows that $\mu_{\epsilon,R}$ is uniformly bounded, by Lemma \ref{EigenDecay}, the eigenfunction $v_{\epsilon,R}$ satisfies
\begin{equation*}
|v_{\epsilon,R}|\leq C e^{-\delta r}\text{ in }\Sigma_{\epsilon}
\end{equation*}
for some constant $C>0$ independent of small $\epsilon$ and large $R$. Then the local elliptic estimates give that
\begin{equation}\label{EigenC2ExpDecay}
|D^2 v_{\epsilon,R}|+|Dv_{\epsilon,R}|+|v_{\epsilon,R}|\leq Ce^{-\delta r}\text{ in }\Sigma_{\epsilon},
\end{equation}
where $C>0$ is independent of $\epsilon$ and $R$.

In the following, we omit the subscripts $\epsilon$ and $R$. Denote $\tilde v:=\chi v$ and decompose $\tilde v$ as
\begin{equation*}
\tilde v=\chi k_1(\ts,\tth)\mathcal T_1(t_1,t_2)+\chi k_2(\ts,\tth)\mathcal T_2(t_1,t_2)+\tilde v^\perp,
\end{equation*}
where for $j=1,2$,
\begin{equation*}
k_j(\ts,\tth)=\frac{\int_{\R^2}\langle \tilde v(\ts,\tth,t_1,t_2), \mathcal T_j(t_1,t_2)\rangle dt_1 dt_2}{\int_{\R^2}\chi|\mathcal T_1(t_1,t_2)|^2 dt_1 dt_2}    
\end{equation*} 
and satisfies
\begin{equation*}
|\nabla_{\Gamma_\epsilon}^2k_j|+|\nabla_{\Gamma_\epsilon}k_j|+|k_j|\leq C
\end{equation*}
for some $C>0$ independent of $R$ and $\epsilon$, due to \eqref{EigenC2ExpDecay}. Then by our decomposition,
\begin{equation*}
\int_{\R^2}\langle \tilde v^\perp(\ts,\tth,t_1,t_2), \mathcal T_j(t_1,t_2)\rangle dt_1 dt_2=0,\text{ for }j=1,2\text{ and }(\ts,\tth)\in \Gamma_{\epsilon}^R.
\end{equation*}
In the following, we will show that $\tilde v^\perp$ is small.

By definition, $\tilde v$ satisfies
\begin{equation*}
\mathbb L(\tilde v;\psi, A)-\epsilon^2\mu\tilde v=-v\Delta \chi -2\nabla\chi\cdot \nabla v+
\begin{pmatrix}
2i (A \cdot d\chi) v\\
0
\end{pmatrix}
=:E_{\epsilon}.
\end{equation*}
Then \eqref{EigenC2ExpDecay} implies that
\begin{equation*}
|E_{\epsilon}|\leq C\epsilon^3 e^{-\delta r}\rho^{-6}
\end{equation*}
for some $\delta>0$. Write
\begin{equation*}
\mathbb L(\tilde v;\psi,A)=\mathbb L_0(\tilde v)+\mathcal G(\tilde v),
\end{equation*}
where $\mathbb L_0(\tilde v)$ is the main contribution of $\mathbb L(;\psi,A)$:
\begin{equation*}
    \mathbb L_0(\tilde v):=
    \begin{pmatrix}
        -\Delta_{\R^2,A^{(1)}}\tilde\xi+2i\epsilon\rho^{-2}\cos(2\epsilon\ts)\tilde\xi_{\tth}-\Delta_{\Ge}\tilde \xi +2i\langle \tilde B,\nabla_{A^{(1)}}\psi^{(1)}\rangle+\frac{\lambda-1}{2}(\psi^{(1)})^2\bar{\tilde \xi}+(\lambda+\frac12)|\psi^{(1)}|^2\tilde \xi-\frac\lambda 2\tilde \xi\\
        \Delta_{H,\R^2} \tilde B+\Delta_{H,\Ge} \tilde B+ 2Im(\overline{\nabla_{A^{(1)}}\psi^{(1)}}\cdot \tilde \xi)+\tilde B|\psi^{(1)}|^2
    \end{pmatrix}
\end{equation*}
for $\tilde v=(\tilde \xi, \tilde B)$ and
\begin{equation*}
\mathcal G(\tilde v):=\mathbb L(\tilde v;\psi,A)-\mathbb L_0(\tilde v).
\end{equation*}
We see that $\tilde v$ satisfies
\begin{equation*}
\mathbb L_0(\tilde v)+\mathcal{G}(\tilde v)-\mu\rho^{-6}\tilde v=E_{\epsilon}\text{ in }\Sigma_{\epsilon}.
\end{equation*}
We can extend $\tilde v$ and $E_{\epsilon}$ as zero outside $\Sigma_{\epsilon}$ to get an equation in the entire $\Ge^R \times \R^2$:
\begin{equation}\label{tildevEqn}
\mathbb L_0(\tilde v)+\tilde \chi\mathcal{G}(\tilde v)-\hat \mu\epsilon^2\rho^{-6}\tilde v=E_{\epsilon},
\end{equation}
where
\begin{equation*}
\begin{cases}
    \tilde \chi=1&\text{ if }a^2+b^2\leq r_\epsilon+1,\\
    \tilde \chi=0&\text{ if }a^2+b^2\geq r_\epsilon+2.
\end{cases}
\end{equation*}
We can also derive the equation of $\tilde v^\perp$:
\begin{equation}\label{tveqnentire}
\mathbb L_0(\tilde v^\perp)-\hat \mu\rho^{-6}\epsilon^2 \tilde v^\perp=-\mathbb 
L_0 (k_1\mathcal T_1+k_2\mathcal T_2)+E_{\epsilon}-\tilde \chi \mathcal G(\tilde v)+\hat\mu\epsilon^2\rho^{-6}(k_1\mathcal{T}_1+k_2\mathcal T_2)\text{ in }\Ge\times \R^2.
\end{equation}
Similar computation as in Proposition \ref{TestKernel} shows that
\begin{equation*}
\begin{aligned}
    &\int_{\R^2}\langle \mathbb L_0 (k_1\mathcal T_1+k_2\mathcal T_2)+\tilde\chi \mathcal{G}(k_1\mathcal T_1+k_2\mathcal T_2), \mathcal T_1 \rangle dt_1dt_2\m+\int_{\R^2}\langle \mathbb L_0 (k_1\mathcal T_1+k_2\mathcal T_2)+\tilde\chi \mathcal{G}(k_1\mathcal T_1+k_2\mathcal T_2), \mathcal T_2\rangle dt_1dt_2\n\\
    =&-L_{\Ge}N\int_{\R^2}|\mathcal T_1|^2+O(\epsilon \rho^{-4}|\nabla_{\Ge}^2 k|)+O(\epsilon^2 \rho^{-4}|\nabla_{\Ge}k|)+O(\epsilon^3 \rho^{-4}|k|),
\end{aligned}
\end{equation*}
where $N=k_1 \m+k_2 \n$.

Testing \eqref{tveqnentire} against $\mathcal T_1$ and $\mathcal T_2$, one can see that $N$ satisfies
\begin{equation*}
\begin{aligned}
&L_{\Ge}N+\hat \mu\epsilon^2 \rho^{-6}N+O(\epsilon \rho^{-4}|\nabla_{\Ge}^2 k|)+O(\epsilon^2 \rho^{-4}|\nabla_{\Ge}k|)+O(\epsilon^3 \rho^{-4}|k|)\\
=&O(\epsilon^3\rho^{-6})+\frac{\int_{\R^2}\langle \tilde\chi \mathcal G(v^\perp), \mathcal T_1\rangle }{\int_{\R^2}|\mathcal T_1|^2} \m+\frac{\int_{\R^2}\langle \tilde\chi \mathcal G(v^\perp), \mathcal T_2\rangle }{\int_{\R^2}|\mathcal T_1|^2}\n.    
\end{aligned}
\end{equation*}

If  $N(\ts,\tth)$ is regarded as a normal vector field $\tilde N(s,\theta):=\tilde k_1(s,\theta) \m+\tilde k_2(s,\theta) \n$ on $\Gamma$, then $\tilde N$ satisfies
\begin{equation}\label{tildeNEqn}
L_{\Gamma}\tilde N+\hat \mu \rho^{-6}\tilde N+O(\epsilon \rho^{-4}|)(\nabla_{\Gamma}^2 \tilde N|+|\nabla_{\Gamma}\tilde N|+|\tilde N|)=O(\epsilon\rho^{-6})+\frac{\int_{\R^2}\langle \tilde\chi \mathcal G(v^\perp), \mathcal T_1\rangle }{\epsilon^2\int_{\R^2}|\mathcal T_1|^2} \m+\frac{\int_{\R^2}\langle \tilde\chi \mathcal G(v^\perp), \mathcal T_2\rangle }{\epsilon^2\int_{\R^2}|\mathcal T_1|^2}\n.
\end{equation}

Once we have shown that the left-hand side is of $O(\epsilon)$, and up to a subsequence, $\tilde N_{\epsilon, R}$ converges to some nontrivial bounded solution $\tilde N_{0,\infty}$ of
\begin{equation*}
L_{\Gamma}\tilde N_{0,+\infty}+\hat \mu_0\rho^{-6}\tilde N_{0,+\infty}=0\text{ in }\Gamma
\end{equation*}
for some $\mu_0\leq 0$ and differs from the known bounded Jacobi fields, this contradicts with the stability and nondegeneracy of $\Gamma$. Hence, it suffices to show that for $j=1,2$,
\begin{equation}\label{mathcalGepsilon3}
\int_{\R^2}\langle \tilde\chi \mathcal G(v^\perp), \mathcal T_j\rangle=O(\epsilon^3).    
\end{equation}
To see this, we need a refined estimate of $v^\perp$. Let us decompose $\tilde v^{\perp}=\tilde v^{\perp}_1+\tilde v^\perp_2$, where $v^{\perp}_1$ satisfies
\begin{equation*}
\begin{aligned}
\mathbb L_0(\tilde v^\perp_1)-\hat \mu\epsilon^2\rho^{-6}\tilde v^\perp_1&=2\epsilon \rho^{-5}[t_1\rho^{-2}(k_{1\ts\ts}\mathcal{T}_1+k_{2\ts\ts}\mathcal{T}_2)- t_1(k_{1\tth\tth}\mathcal{T}_1+k_{2\tth\tth}\mathcal{T}_2)+2 t_2(k_{1\ts\tth}\mathcal{T}_1+k_{2\ts\tth}\mathcal{T}_2)]\\
&+O(\epsilon^2\rho^{-2}(k_{1\ts}\nabla_{\R^2}\mathcal T_1+k_{2\ts}\nabla_{\R^2}\mathcal T_2+k_{1\tth}\nabla_{\R^2}\mathcal T_1+k_{2\tth}\nabla_{\R^2}\mathcal T_2))=:\mathcal H_1,
\end{aligned}
\end{equation*}
and the right-hand side $\mathcal{H}_1$ consists of large terms in the expansion of $\mathbb 
L_0 (k_1\mathcal T_1+k_2\mathcal T_2)+\mathcal{G}(k_1\mathcal T_1+k_2\mathcal T_2)$ which are orthogonal to both $\mathcal T_1$ and $\mathcal T_2$. See $Q_1$, $Q_2$, $Q_3$ and $Q_5$ in the proof of Proposition \ref{TestKernel} for exact expressions.

Lemma \ref{lineartheory} for $\mathbb L_0$ implies the existence of such $\tilde v_1^\perp$ with
\begin{equation*}
\int_{\R^2}\langle \tilde v_1^\perp, \mathcal T_j\rangle=0,\ j=1,2
\end{equation*}
and
\begin{equation*}
\|\tilde v^\perp_1 \|_{2,2,p,\delta}\leq C \|\mathcal H_1\|_{0,4,p,\delta}\leq C\epsilon.   
\end{equation*}
As a consequence, we have
\begin{equation*}
\|\tilde \chi \mathcal G(\tilde v^\perp_1)\|_{0,4,p,\delta}\leq C\epsilon^2.
\end{equation*}
We see that $\tilde v^\perp _2$ is a solution of
\begin{equation*}
\mathbb L_0(\tilde v^\perp_2)-\hat \mu\epsilon^2\rho^{-6}\tilde v^\perp_2 +\tilde \chi \mathcal G(\tilde v^\perp_2 )=\mathcal H_1-\mathbb 
L_0 (k_1\mathcal T_1+k_2\mathcal T_2)+E_{\epsilon}-\tilde \chi \mathcal G(\tilde v^\perp_1)+\hat\mu\epsilon^2\rho^{-6}(k_1\mathcal{T}_1+k_2\mathcal T_2)
\end{equation*}
with
\begin{equation*}
\int_{\R^2}\langle \tilde v_2^\perp, \mathcal T_j\rangle=0,\ j=1,2.
\end{equation*}

Note that Proposition \ref{TestKernel}
implies that the right-hand side can be written as
\begin{equation*}
(L_{\Ge}N)_1\mathcal T_1+(L_{\Ge}N)_2\mathcal T_2+F
\end{equation*}
for some
\begin{equation*}
\|F\|_{0,4,p,\delta}\leq C\epsilon^2.
\end{equation*}
 Lemma \ref{lineartheory} for $\mathbb L_0$ tells us  that
\begin{equation}\label{tildev2Est}
\|\tilde v^\perp_2\|_{2,2,p,\delta}\leq C\epsilon^2,
\end{equation}
which implies that
\begin{equation*}
\|\tilde \chi \mathcal G(\tilde v^\perp_2)\|_{0,4,p,\delta}\leq C\epsilon^3.
\end{equation*}

Applying Lemma \ref{JacobiEigenEst} to \eqref{tildeNEqn}, we see that
\begin{equation*}
\|\nabla_{\Gamma}^2 \tilde N\|_{p,4}+\|\nabla_{\Gamma} \tilde N\|_{p,3}\leq C+C\|\tilde N\|_{L^{\infty}}\leq C
\end{equation*}
and equivalently,
\begin{equation*}
\|\nabla_{\Ge}^2  N\|_{p,4}+\|\nabla_{\Ge} N\|_{p,3}\leq C\epsilon.
\end{equation*}
This further improves the estimates of $\mathcal H_1$ to
\begin{equation*}
\|\mathcal H_1\|_{0,4,p,\delta}\leq C\epsilon^2.    
\end{equation*}

Consequently, the estimate of $\tilde v^\perp$ can also be improved to
\begin{equation}\label{tildev1Est}
\|\tilde v^\perp_1 \|_{2,2,p,\delta}\leq C \|\mathcal H_1\|_{0,4,p,\delta}\leq C\epsilon^2
\end{equation}
and therefore we get the following  estimate for $\chi \mathcal G(\tilde v^\perp_1)$:
\begin{equation*}
\|\tilde \chi \mathcal G(\tilde v^\perp_1)\|_{0,4,p,\delta}\leq C\epsilon^3.
\end{equation*}
From this we obtain \eqref{mathcalGepsilon3}.

Recall that $\tilde v_{\epsilon,R}=k_{1,\epsilon,R}\mathcal{T}_1+k_{2,\epsilon,R}\mathcal{T}_2+ \tilde v_{\epsilon,R}^\perp$, where $\tilde v_{\epsilon,R}^\perp=O(\epsilon \rho^{-2}e^{-\delta r})$ and $\| v_{\epsilon,R}\|_{L^{\infty}}=1$ by our assumption. Since $\tilde v_{\epsilon,R}$ decays exponentially, the supremum is attained in $\Sigma_\epsilon$ if $\epsilon$ is sufficiently small. This forces
\begin{equation*}
    \|\tilde N_{\epsilon,R}\|_{L^{\infty}(\Gamma)}\geq c_0
\end{equation*}
for some $c_0>0$ independent of $R$. 

By the uniform $C^1$ bound \eqref{EigenC2ExpDecay}, we can pass the limit $R\to \infty$ in \eqref{tildevEqn} and \eqref{tildeNEqn}. We see that $v_{\epsilon,R}\to v_{\epsilon}$ and
$\tilde N_{\epsilon, R}\to \tilde N_{\epsilon}$ for some nontrivial $v_{\epsilon}$ and $\tilde N_{\epsilon}=\tilde k_{1,\epsilon}\m+\tilde k_{2,\epsilon}\n$ solving
\begin{equation}\label{tildevepsiloneqn}
    \mathbb L_0(\tilde v_{\epsilon})+\mathcal{G}(\tilde{v}_{\epsilon})-\hat \mu_{\epsilon}\epsilon^2\rho^{-6}\tilde v_{\epsilon}=E_{\epsilon}\text{ in }\R^4
\end{equation}
and
\begin{equation*}
L_{\Gamma}\tilde N_{\epsilon}+\hat \mu_\epsilon\rho^{-6}\tilde N_{\epsilon}=O(\epsilon\rho^{-4})\text{ in }\Gamma
\end{equation*}
respectively for some $\hat \mu_\epsilon<0$. Moreover, they satisfy
\begin{equation}\label{tildevepsilondecom}
\tilde v_{\epsilon}=\tilde k_{1,\epsilon}\mathcal{T}_1+\tilde k_{2,\epsilon}\mathcal{T}_2+\tilde v_\epsilon^\perp\text{ in }\Sigma_\epsilon
\end{equation}
and
\begin{equation}\label{Nepsilonest}
\|\nabla_{\Gamma}^2 \tilde N_{\epsilon}\|_{p,4}+\|\nabla_{\Gamma} \tilde N_{\epsilon}\|_{p,3}\leq C[\|\tilde N_{\epsilon}\|_{L^\infty(\rho<3R_0)}+O(\epsilon)].
\end{equation}

Next, we test \eqref{tildevepsiloneqn} against $Z_j$, $j=1,\dots, 6$. Recall that $Z_j$'s are bounded kernels of $\mathbb L(\cdot;\psi,A)$ introduced in \eqref{modifiedkernel}. Then we have
\begin{equation*}
\int_{\R^4}\langle v_{\epsilon}, Z_j\rangle \rho^{-6} dx=O(\epsilon)\text{ for }j=1,\dots, 6. 
\end{equation*}
and hence,
\begin{equation*}
\int_{\Sigma_\epsilon}\langle \tilde v_{\epsilon}, Z_j\rangle \rho^{-6} dx=O(\epsilon)\text{ for }j=1,\dots, 6. 
\end{equation*}

Observe that $Z_j$ can be written as
\begin{equation*}
Z_j=(N_j\cdot\m)\mathcal{T}_1+(N_j\cdot \n)\mathcal{T}_2+O(\epsilon \rho^{-2} e^{-\delta r}).
\end{equation*}
Here we recall the bounded Jacobi fields are given in \eqref{Jacobifields}. Then in view of the decomposition \eqref{tildevepsilondecom}, the integral above passes to
\begin{equation}\label{tildeNepsortho}
\int_{\Gamma}\langle \tilde N_\epsilon, N_j\rangle \rho^{-6} dvol_{\Gamma}=O(\epsilon),\ j=1,\dots, 6.
\end{equation}

Due to the estimate \eqref{Nepsilonest}, up to a subsequence, $\tilde N_\epsilon$ converges locally uniformly to a nontrivial bounded solution $\tilde N$ to
\begin{equation*}
L_\Gamma \tilde N+\hat \mu\rho^{-6}\tilde N=0\text{ in }\Gamma
\end{equation*}
with $\hat \mu\leq 0$. Furthermore, by the stability of $\Gamma$, we see that $\hat \mu=0$.

However, passing $\epsilon\to 0$ in \eqref{tildeNepsortho}, we deduce
\begin{equation*}
\int_{\Gamma}\langle \tilde N, N_j\rangle \rho^{-6} dvol_{\Gamma}=0,\ j=1,\dots, 6,
\end{equation*}
a contradiction to the non-degeneracy of $\Gamma$. Thus we finish the proof of the stability part in Theorem \ref{main}.

We proceed to prove the non-degeneracy of $U_\epsilon=(\psi_\epsilon, A_\epsilon)$. Assume to the contrary that $\mathbb L(\cdot;\psi,A)$ admitted another bounded kernel $Z_7$ with
\begin{equation*}
\int_{\R^4}\langle Z_7, Z_j\rangle \rho^{-6} dx=0,\ j=1,\dots, 6.
\end{equation*}
Note that the arguments above also apply to $Z_7$ so that
\begin{equation*}
    Z_7=(N_{7,\epsilon}\cdot\m) \mathcal{T}_1+(N_{7,\epsilon}
    \cdot \n)\mathcal{T}_2+Z_7^\perp\text{ in }\Sigma_{\epsilon}
\end{equation*}
for some function $Z_7^\perp=O(\epsilon\rho^{-2}e^{-\delta r})$ which is orthogonal to both $\mathcal T_1$ and $\mathcal T_2$ and normal vector field $N_{7,\epsilon}$ converging locally uniformly to a bounded Jacobi field $N_7$ on $\Gamma$ that is orthogonal to $N_j$, $j=1,\dots,7$. This is a contradiction to the non-degeneracy of $\Gamma$. The proof of non-degeneracy of $U_\epsilon$ and Theorem \ref{main} is thus completed.


\begin{thebibliography}{99}  

\bibitem{AC} L. Ambrosio, X. Cabr\'e, \textit{Entire solutions of semilinear elliptic equations in $\R^3$ and
a conjecture of De Giorgi.} Journal Amer. Math. Soc. 13 (2000), 725–739.

\bibitem{ArePac}C. Arezzo, F. Pacard, \textit{Complete, embedded, minimal n-dimensional submanifolds in $\mathbb C^n$.} Comm. Pure Appl. Math. 56 (2003) 283–327. https://doi.org/10.1002/cpa.10060.


\bibitem{BM1}M. Badran, M. Del Pino,  \textit{Solutions to the magnetic Ginzburg-Landau equations concentrating on codimension-2 minimal submanifolds.} Vietnam J. Math. 52 (2024), no. 4, 967–984.

\bibitem{BM2}M. Badran, M. Del Pino, \textit{Solutions of the Ginzburg-Landau equations concentrating on codimension-2 minimal submanifolds.} J. Lond. Math. Soc. (2) 109 (2024), no. 1, Paper No. e12851, 31 pp.

\bibitem{BM3}M. Badran, M. Del Pino,  \textit{Entire solutions to 4 dimensional Ginzburg-Landau equations and codimension 2 minimal submanifolds.} Adv. Math. 435 (2023), part A, Paper No. 109365, 73 pp.

\bibitem{BerChen} M. S. Berger, Y. Chen, \textit{Symmetric vortices for the nonlinear Ginzburg–Landau equations of superconductivity,and the nonlinear desingularization phenomenon.} J. Funct. Anal. 82, 259–295 (1989).

\bibitem{BBH}F. Bethuel, H. Brezis,  F. Hélein, \textit{Ginzburg–Landau vortices, Progress in nonlinear differential equations and their applications}. vol. 13, Birkhäuser Boston, Inc., Boston, MA, 1994.

\bibitem{B}S. B. Bradlow, \textit{Vortices in holomorphic line bundles over closed K\"ahler manifolds.} Commun. Math. Phys., vol. 135, no. 1 (1990), pp. 1–17.

\bibitem{Br}S. Brendle, \textit{On solutions to the Ginzburg-Landau equations in higher dimensions.} arXiv: Math/0302070. (2003). http://arxiv.org/abs/math/0302070.

\bibitem{CT1} X. Cabr\'e, J. Terra, \textit{Saddle-shaped solutions of bistable diffusion equations in all of $R^{2m}$.} J. Eur. Math. Soc. (JEMS) 11 (2009), no. 4, 819–843.

\bibitem{CT2}  X. Cabr\'e, J. Terra, \textit{Qualitative properties of saddle-shaped solutions to bistable diffusion equations.}
Comm. Partial Differential Equations 35 (2010), no. 11, 1923–1957.

\bibitem{Ca}  X. Cabr\'e, \textit{Uniqueness and stability of saddle-shaped solutions to the Allen-Cahn equation.} J. Math. Pures Appl. (9) 98 (2012), no. 3, 239–256.


\bibitem{CM1} O. Chodosh, C. Mantoulidis, \textit{ Minimal surfaces and the Allen-Cahn equation on 3-manifolds: index, multiplicity, and curvature estimates. }  Ann. of Math. (2) 191 (2020).  https://doi.org/10/gn7c3j.


\bibitem{CM2} O. Chodosh, C. Mantoulidis, \textit{The p-widths of a surface.} Publ. Math. IHES. 137 (2023) 245–342. https://doi.org/10.1007/s10240-023-00141-7.

\bibitem{DeG} E. De Giorgi, \textit{Convergence problems for functionals and operators.} Proc. Int. Meeting on Recent Methods in Nonlinear Analysis (Rome, 1978), 131–188, Pitagora, Bologna (1979).


\bibitem{PHP} G. De Philippis, A. Halavati, A. Pigati, \textit{Decay of excess for the abelian Higgs model.}  arXiv:2405.13953.

\bibitem{Philippis-Pigati}  G. De Philippis, A. Pigati, \textit{Non‐degenerate minimal submanifolds as energy concentration sets: A variational approach.} Comm. Pure Appl. Math. 2024;77:3581–3627.


\bibitem{PinoKow1} M. Del Pino, M. Kowalczyk, J. Wei, \textit{On De Giorgi’s conjecture in dimension $N\geq 9$.} Ann. of Math. (2) 174 (2011), no. 3, 1485–1569.

\bibitem{PinoKow2} M. Del Pino, M. Kowalczyk, J. Wei, \textit{Entire solutions of the Allen-Cahn equation and complete embedded minimal surfaces of finite total curvature in 
$\R^3$.} J. Diff. Geom. 93 (1) 67–131. https://doi.org/10.4310/jdg/1357141507.

\bibitem{Dey} A. Dey, \textit{A comparison of the Almgren-Pitts and the Allen-Cahn min-max theory.} Geom. Funct. Anal. 32 (2022), 980–1040.

\bibitem{Donaldson}S. K. Donaldson, P. B. Kronheimer,  \textit{The geometry of four-manifolds. Oxford Mathematical Monographs.} Oxford Science Publications. The Clarendon Press, Oxford University Press, New York, 1990.




\bibitem{GG1} N. Ghoussoub, C. Gui, \textit{On a conjecture of De Giorgi and some related problems.}
Math. Ann. 311 (1998), 481–491.

\bibitem{GG2} N. Ghoussoub, C. Gui, \textit{On De Giorgi’s conjecture in dimensions 4 and 5.} Ann. of
Math. (2) 157 (2003), no. 1, 313–334.


\bibitem {GG} P. Gaspar, M.A.M. Guaraco, \textit{The Allen-Cahn equation on closed manifolds}. Calc. Var. Partial Differential Equations, 57 (2018) Art. 101, 42.

\bibitem {Gu} M.A.M. Guaraco, \textit{Minmax for phase transitions and the existence of embedded minimal hypersurfaces}. J. Differential Geom., 108 (1) (2018), 91–133.

\bibitem{GS}S. Gustafson, I. M. Sigal, \textit{The stability of magnetic vortices.} Comm. Math. Phys. 212 (2000), no. 2, 257–275.

\bibitem{H1}A. Halavati, \textit{New weighted inequalities on two manifolds.} preprint. 2023.

\bibitem{H2}A. Halavati, \textit{Quantitative stability of Yang–Mills–Higgs instantons in two dimensions.} Arch. Rational Mech. Anal. (2024) 248:88.

\bibitem{HJS}M.-C. Hong, J. Jost,  M. Struwe, \textit{Asymptotic limits of a Ginzburg–Landau type functional.} Geometric analysis and the calculus of variations. Dedicated to Stefan Hilde-brandt on the occasion of his 60th birthday. Cambridge, MA: International Press, 1996, pp. 99–123.


\bibitem{Taubes}A. Jaffe, C. Taubes,  \textit{Vortices and monopoles. Structure of static gauge theories.} Progress in Physics, 2. Birkhäuser, Boston, MA, 1980.


\bibitem{LR}F. H. Lin, T. Rivi\`ere, \textit{Complex Ginzburg-Landau equations in high dimensions and codimension
two area minimizing currents.} Journal of the European Mathematical Society. 1 (1999) 237-311. https://doi.org/10.1007/s100970050008.

\bibitem{LMWW} Y. Liu, X. Ma, J. Wei, W. Wu, \textit{Entire solutions of the magnetic Ginzburg-Landau equation in $\R^4$.}  accepted by Annali Scuola Normale Superiore-Classe De Scienze.

\bibitem{LWW1} Y. Liu, K. Wang, J. Wei, \textit{Global minimizers of the Allen-Cahn equation in dimension $n \geq 8$.} J. Math. Pures Appl. (9) 108 (2017), no. 6, 818–840.

\bibitem{LWW2} Y. Liu, K.  Wang, J. Wei, \textit{ Stability of saddle solutions for the Allen Cahn equation}, preprint 2020.



\bibitem {Muc} D. McDuff, D. Salamon, \emph{Introduction to symplectic topology}. Third edition. Oxford Graduate Texts in Mathematics. Oxford University Press, Oxford, 2017.



\bibitem{PacRiv}F. Pacard, T. Rivière, \textit{Linear and nonlinear aspects of vortices. The Ginzburg-Landau model.} Progress in Nonlinear Differential Equations and their Applications, 39. Birkhäuser Boston, Inc., Boston, MA, 2000.

\bibitem{PW}F. Pacard, J. Wei, \textit{Stable solutions of the Allen–Cahn equation in dimension 8 and minimal cones.} J. Funct. Anal. 264 (5) (2013) 1131–1167.

\bibitem{Parise} D. Parise, A. Pigati, D. Stern, \textit{Convergence of the self-dual U(1)-Yang-Mills-Higgs energies to the (n-2)-area functional.} Comm. Pure Appl. Math. 77(2024), no.1, 670–730.

\bibitem{Pigati-Stern} A. Pigati, D. Stern, \textit{Minimal submanifolds from the abelian Higgs model.} Invent. Math. 223 (2021) 1027–1095. https://doi.org/10.1007/s00222-020-01000-6.

\bibitem{Plo}B. J. Plohr, \textit{The existence, regularity, and behavior of isotropic solutions of classical gauge field theories.} Princeton University, 1980.

\bibitem{Riv} T. Rivi\`ere, \textit{Towards Jaffe and Taubes conjectures in the strongly repulsive limit.} Manuscripta Math. 108, 217–273 (2002). https://doi.org/10.1007/s002290200266.



\bibitem{SS}E. Sandier, S. Serfaty,  \textit{Vortices in the magnetic Ginzburg-Landau model.} Progress in Nonlinear Differential Equations and their Applications, 70. Birkhäuser Boston, Inc., Boston, MA, 2007.



\bibitem{S} O. Savin, \textit{Regularity of flat level sets in phase transitions.} Ann. of Math.(2) 169(2009),
no.1, 41–78.


\bibitem{T2}C. H. Taubes, \textit{On the equivalence of the first and second order equations for
gauge theories.} Commun. Math. Phys., vol. 75, no. 3 (1980), pp. 207–227.


\bibitem{T1}C. H. Taubes, \textit{Arbitrary N-vortex solutions to the first order Ginzburg–Landau
equations.} Commun. Math. Phys., vol. 72, no. 3 (1980), pp. 277–292.

\bibitem{TW}F. Ting, J. Wei, \textit{Multi-vortex non-radial solutions to the magnetic Ginzburg-Landau equations.} Commun. Math. Phys. 317 (2013) 69–97.


\end{thebibliography}
\end{document}